\documentclass[11pt]{amsart}

\usepackage{amsmath,amsthm,amsfonts}
\usepackage{epsfig}
\usepackage{pb-diagram}
\usepackage{enumitem}

\newtheorem{theorem}{Theorem}[section]
\newtheorem{lemma}[theorem]{Lemma}
\newtheorem{corollary}[theorem]{Corollary}

\newtheorem{conjecture}[theorem]{Conjecture}

\newcommand{\dist}{{\mathit dist}}
\newcommand{\diam}{{\mathit diam}}
\newcommand{\depth}{{\mathit depth}}

\title[Separations in symmetric graphs]{Small separations in vertex transitive graphs}

\author[M. DeVos]{Matt DeVos}
\address{Department of Mathematics,% \\
        Simon Fraser University,% \\
        Burnaby, B.C.}
\email{{\tt mdevos@sfu.ca}}
\thanks{Supported in part by an NSERC Discovery Grant.}

\author[B. Mohar]{Bojan Mohar}
\address{Department of Mathematics,% \\
        Simon Fraser University,% \\
        Burnaby, B.C.}
\email{{\tt mohar@sfu.ca}}
\thanks{Supported in part by an NSERC Discovery Grant,
        by the Canada Research Chair Program, and 
        by the ARRS, Research Program P1-0297.}
\thanks{On leave from IMFM \& FMF, Department of Mathematics,
        University of Ljubljana,
        1000 Ljubljana, Slovenia.}

\newcommand{\DEF}[1]{{\em #1\/}}

\newcommand{\ZZ}{\mbox{$\mathbb Z$}}
\newcommand{\ZZZ}{{\mathbb Z}}
\newcommand{\NN}{\mbox{$\mathbb N$}}

\renewcommand\int{\mathop{\hbox{{\rm int}}}}

\begin{document}

\begin{abstract}
Let $k$ be an integer. We prove a rough structure theorem for separations
of order at most $k$ in finite and infinite vertex transitive graphs.  Let
$G = (V,E)$ be a vertex transitive graph, let $A \subseteq V$ be a finite
vertex-set with $|A| \le \frac12 |V|$ and
$|\{ v \in V \setminus A : \mbox{$u \sim v$ for some $u \in A$} \}|\le k$.
We show that whenever the diameter of $G$ is at least $31(k+1)^2$, either
$|A| \le 2k^3+k^2$, or $G$ has a ring-like structure (with bounded parameters),
and $A$ is efficiently contained in an interval. This theorem may be viewed as a rough characterization, generalizing an earlier result of Tindell, and has applications to the study of product sets and expansion in groups.
\end{abstract}

\maketitle

\section{Overview}

The study of expansion in vertex transitive graphs and in groups divides naturally into
the study of local expansion, or \DEF{connectivity},
and the study of global expansion, or \DEF{growth}.
The expansion properties of a group are those of its Cayley graphs, so vertex
transitive graphs are the more general setting.  Our main result, Theorem \ref{main_thm}, concerns local expansion in vertex transitive graphs, but it is also meaningful for groups, and it has some asymptotic applications.
Theorem \ref{main_thm} can be viewed as a rough characterization of vertex transitive graphs with small separations. It is shown that a separator of order $k$ in a vertex transitive graph separates a set of bounded size unless the graph has a ring-like structure and the separated set is essentially an interval in this structure.
This result reaches far beyond any previous results on separation in vertex transitive graphs, since $k$ does not need to be bounded in terms of the degree of the graphs.
Several corollaries (\ref{thm:3}, \ref{cor:1.10}, \ref{cor:17}, and \ref{cor:19}) indicate some of diverse possibilities of applications of our main theorem.

Graphs and groups appearing in this paper may be finite or infinite. Nevertheless,
it is assumed that graphs are locally finite and groups are finitely generated.

We continue with a tour of some of the important theorems in expansion.

\medskip

\subsection*{Local Expansion in Groups}

This is the study of small sum sets or small product sets.
Let ${\mathcal G}$ be a (multiplicative) group and let $A,B \subseteq {\mathcal G}$.
The main questions of interest here are lower bounds on $|AB|$, and in the case when
$|AB|$ is small, finding the structure of the sets $A$ and $B$.
The first important result in this area was proved by Cauchy and (independently) Davenport.  For every  positive integer $n$, we let $\ZZ_n = \ZZ / n\ZZ$.

\begin{theorem}[Cauchy \cite{Cau}, Davenport \cite{Dav}]
Let $p$ be prime and let $A,B \subseteq \ZZ_p$ be nonempty.  Then
$|A + B| \ge \min\{ p, |A| + |B| - 1 \}$.
\end{theorem}

This theorem was later refined by Vosper who found the structure of all
$A,B \subseteq \ZZ_p$ ($p$ prime) for which $|A+B| < |A| + |B|$.  Before stating
his theorem, note that in any finite group ${\mathcal G}$ we must have $AB = {\mathcal G}$
whenever $|A| + |B| > |{\mathcal G}|$ by the following pigeon hole argument:
$\{ a^{-1}g : a \in A \} \cap B \neq \emptyset$ for every $g \in {\mathcal G}$.  For simplicity, we
have excluded this uninteresting case below.

\begin{theorem}[Vosper \cite{Vos}]
Let $p$ be a prime and let $A,B \subseteq \ZZ_p$ be nonempty.
If\/ $|A+B| < |A| + |B| \le p$, then one of the following holds:

\begin{enumerate}[label={\rm (\roman{*}) \ }, ref={\rm (\arabic{*})}]
\item
$|A| = 1$ or\/ $|B| = 1$.
\item
There exists $g \in \ZZ_p$ so that $B = \{g - a : a \in \ZZ_p \setminus A \}$,
and $A+B = \ZZ_p \setminus \{g\}$.
\item
$A$ and $B$ are arithmetic progressions with a common difference.
\end{enumerate}
\end{theorem}

Analogues of the Cauchy-Davenport theorem and Vosper's theorem for abelian and general groups were found by Kneser \cite{Kn}, Kempermann \cite{Kemperman}, and recently by
DeVos \cite{DeVos}, and DeVos, Goddyn and Mohar \cite{DGM}.

In abelian groups, we have powerful theorems
which yield rough structural information when a finite subset $A \subseteq {\mathcal G}$ satisfies $|A+A| \le c |A|$ for a fixed constant $c$.
Freiman \cite{Freiman} proved such a theorem when ${\mathcal G} = \ZZ$ and this has
recently been extended to all abelian groups by Green and Ruzsa \cite{GreenRuzsa}.
Despite this progress, there is still relatively little known in terms of rough structure of sets with small product in general groups.
The following corollary of our main theorem is a small step in this direction.

\begin{theorem}
\label{thm:3}
Let\/ ${\mathcal G}$ be an infinite group, let $B=B^{-1} \subseteq {\mathcal G}$ be a finite generating set containing the identity element of ${\mathcal G}$, and let
$A \subseteq {\mathcal G}$ be a finite subset of ${\mathcal G}$.
If $|BA| < |A| + \tfrac{1}{2}|A|^{\frac{1}{3}}$, then ${\mathcal G}$ has a finite normal subgroup $N$ so that ${\mathcal G}/N$ is either cyclic or dihedral.
Furthermore, $|N|< \frac{1}{4} |A|^{1/3}$.
\end{theorem}

The proof is given towards the end of the paper in Section~\ref{sect:proofs}.
Our result also applies to finite groups, but for this it requires an assumption which is more natural in the context of graphs and will be discussed in the sequel.

\bigskip

\subsection*{Local Expansion in Graphs}

Before we begin our discussion of expansion in graphs, we will need to introduce some notation.  If $G$ is a graph and $X \subseteq V(G)$, we let
$\delta X = \{ uv \in E(G) : \mbox{$u \in X$ and $v \not\in X$} \}$ and we call any
set of edges of this form an \DEF{edge-cut}.  We let
$\partial X = \{ v \in V(G) \setminus X : \mbox{$uv \in E(G)$ for some $u \in X$} \}$
and we call $\partial X$ the \DEF{boundary} of $X$.
Similarly, if $\vec{G}$ is a directed graph and
$X \subseteq V(\vec{G})$ we let
$\delta^+(X) = \{ (u,v) \in E(\vec{G}) : \mbox{$u \in X$ and $v \not\in X$} \}$ and
$\delta^-(X) = \delta^+( V(G) \setminus X)$, and we let
$\partial^+(X) = \{ v \in V(G) \setminus X : \mbox{$(u,v) \in E(G)$ for some $u \in X$} \}$ and
$\partial^-(X) = \{ v \in V(G) \setminus X : \mbox{$(v,u) \in E(G)$ for some $u \in X$} \}$.
\DEF{Expansion in graphs} is the study of the behavior of the cardinalities of
$\delta X$ and $\partial X$.  We will be particularly interested in the case when these parameters are small.  Next we introduce the types of graphs we will be most interested in.

Again we let ${\mathcal G}$ denote a multiplicative group.  For every $A \subseteq {\mathcal G}$, we define the Cayley digraph ${\mathit Cay}({\mathcal G},A)$ to be
the directed graph (without multiple edges) with vertex-set ${\mathcal G}$ and $(x,y)$ an arc if
$y \in Ax$.  Using this definition, the group ${\mathcal G}$ has a natural (right) transitive action on $V(G)$ which preserves incidence.   If $1 \in A$, and $B \subseteq {\mathcal G}$, then the product set $AB$ is the disjoint union of $B$
and $\partial^+(B)$.  This observation allows us to rephrase problems about small product sets in groups as problems concerning sets with small boundary in Cayley digraphs.

Next we overview known bounds on the boundary of finite sets in vertex transitive graphs.  These theorems are usually stated only for finite graphs, but more general
versions stated below follow from the same arguments.

%\begin{theorem} Let $G = (V,E)$ be a connected $d$-regular vertex transitive graph, let
%$\emptyset \neq A \subseteq V$ be finite, and assume that $|A| \le \frac{|V|}{2}$.

%\begin{enumerate}[label={\rm (\arabic{*}) \ }, ref={\rm \arabic{*}}]
%\item $|\delta A| \ge d$ \quad {\rm (Mader \cite{Mad71a})}.
%\item $|\partial A| \ge \min \{ |V \setminus A|, \frac{2}{3}(d+1) \}$
%    \quad {\rm (Mader \cite{Mad71b}, Watkins \cite{Wat}).}
%\item $|\partial A| \ge \frac{|A|}{ \diam(A) + 1}$ \quad {\rm (Babai, Szegedy \cite{BSz}) }.
%\item If $|A| \ge 2$ and $|\delta A| = d$, then {\rm (i)} or {\rm (ii)} holds \quad {\rm (Tindell \cite{Tin})}
%    \begin{enumerate}[label={\rm (\roman{*}) \ }, ref={\rm (\arabic{*})}]
%    \item   There is a block of imprimitivity which is a clique of size $d$.
%    \item   $d=2$ (so $G$ is a cycle).
%    \end{enumerate}
%\item If $|A| \ge \frac{1}{3}(d+1)$ and $|\delta A| \le \frac{2}{9}(d+1)^2$, then $G$ has a block of
%    imprimitivity of size $\le \frac{2}{9}(d+1)^2$ \quad {\rm (van den Heuvel, Jackson \cite{HJ})}.
%\end{enumerate}
%\end{theorem}

\begin{theorem}[Mader \cite{Mad71a}]
\label{edge-con_thm}
If $G$ is a connected $d$-regular vertex transitive graph and\/ $\emptyset \neq A \subset V(G)$
is finite, then $|\delta A| \ge d$.
\end{theorem}

\begin{theorem}[Mader \cite{Mad71b}, Watkins \cite{Wat}]
\label{vertex-con_thm}
If $G$ is a connected $d$-regular vertex transitive graph and $\emptyset \neq A \subset V(G)$ is finite and satisfies $A \cup \partial A \neq V(G)$, then $|\partial A| \ge \frac{2}{3}(d+1)$.
\end{theorem}

\begin{theorem}[Hamidoune \cite{Ham}]
\label{digraph-con_thm}
If $G$ is a connected vertex transitive directed graph with outdegree $d$ and
$\emptyset \neq A \subset V(G)$ is finite and satisfies $A \cup \partial^+(A) \neq V(G)$,
then $|\partial^+(A)| \ge \frac{d+1}{2}$.
\end{theorem}

For the next result, recall that a partition $\sigma$ (whose parts are called
\DEF{blocks}) of vertices of a vertex transitive graph $G$ is said to be
a \DEF{system of imprimitivity} if for every automorphism $\varphi$ of $G$
and every block $B\in \sigma$, the set $\varphi(B)$ is another block of $\sigma$.
In that case the blocks of $\sigma$ are also called \DEF{blocks of imprimitivity}.
Having a system of imprimitivity $\sigma$, we define the quotient graph
$G^\sigma$ whose vertices are the blocks of imprimitivity, and two such blocks
$B,B'$ are adjacent if there exist $u\in B$ and $u'\in B'$ that are adjacent in $G$.

The next theorem may be viewed as a refinement of Mader's theorem which gives a structural result for graphs which have small edge-cuts.
Let us recall that a vertex-set $B$ in a graph $G$ is called a \DEF{clique}
if every two vertices in $B$ are adjacent.

\begin{theorem}[Tindell \cite{Tin}]
Let $G$ be a finite connected $d$-regular vertex transitive graph.  If there exists
$X \subseteq V(G)$ with $|X|,|V(G) \setminus X| \ge 2$ so that\/ $|\delta X| = d$, then one of the following holds:

\begin{enumerate}[label={\rm (\roman{*}) \ }, ref={\rm (\arabic{*})}]
\item   There is a system of imprimitivity whose blocks are cliques of order~$d$.
\item   $d=2$ (so\/ $G$ is a cycle).
\end{enumerate}
\end{theorem}

The following recent theorem of van den Heuvel and Jackson gives a rough analogue of Tindell's
result under the assumption that $G$ has a fairly small edge-cut with sufficiently many vertices on either side.

\begin{theorem}[van den Heuvel, Jackson \cite{HJ}]
If $G$ is a finite connected $d$-regular vertex transitive graph and there exists a set $S \subseteq V(G)$ with
$\frac{1}{3}(d+1) \le |S| \le \frac{|V(G)|}{2}$ and
$|\delta S| < \frac{2}{9}(d+1)^2$, then there is
a block of imprimitivity that has less than $\tfrac{2}{9}(d+1)^2$ vertices.
\end{theorem}

Our main theorem may be viewed as a rough characterization that generalizes
Tindell's result, but without any assumptions relating to the degree of the graph and with an added assumption that the diameter
of the graph is large.  Before introducing the theorem we will require
some further definitions.
For $A \subseteq V(G)$, we denote by $G[A]$ the subgraph of $G$ induced on the
vertices in $A$. We define the \DEF{diameter} of $A$, denoted $diam(A)$, as the supremum of
$\dist_G(x,y)$ over all $x,y \in A$.  The \DEF{diameter} of $G$, denoted $\diam(G)$, is defined to be $\diam(V(G))$.
If $A$ is a proper subset of $V(G)$, then the \DEF{depth} of a vertex $v$ in $A$ is
$\dist_G(v, V(G) \setminus A)$.  The \DEF{depth} of $A$, denoted
$\depth(A)$, is the supremum over
all vertices $v$ in $A$ of the depth of $v$ in $A$.

If $S$ is a set, a \DEF{cyclic order on} $S$ is a symmetric relation $\sim$ so
that the corresponding graph is either a circuit, or a two-way-infinite path.
The \DEF{distance} between two elements in $S$ is defined to be the
distance in the corresponding graph, and an \DEF{interval} of $S$ is a finite subset $\{s_1,s_2,\ldots,s_m\} \subseteq S$ with $s_i \sim s_{i+1}$ for every $1 \le i \le m-1$.  A {\it cyclic system} $\vec{\sigma}$ on a graph $G$ is a system
of imprimitivity $\sigma$ on $V(G)$ equipped with a cyclic order (indicated by the arrow) which is preserved by the automorphism group of $G$.  If $s,t$ are positive integers,
we say that $G$ is \DEF{$(s,t)$-ring-like} with respect to the cyclic system $\vec{\sigma}$ if every block of
$\sigma$ has size $s$ and any two adjacent vertices of $G$ are in blocks which
are at distance $\le t$ (in the cyclic order) in $\vec{\sigma}$.

The main result of this paper is:

\begin{theorem}
\label{main_thm}
Let $G$ be a vertex transitive graph, let $A \subseteq V(G)$ be a finite
non-empty set with $|A| \le \tfrac12 |V(G)|$
%and\/ $G[A]$ connected.
such that\/ $G[A\cup\partial A]$ is connected.
Set $k = |\partial A|$ and assume that $diam(G) \ge 31(k+1)^2$.
Then one of the following holds:

\begin{enumerate}[label={\rm (\roman{*}) \ }, ref={\rm (\arabic{*})}]
\item   $\depth(A) \le k$ and\/ $|A| \le 2k^3+k^2$
$($and\/ $G$ is $d$-regular where $d \le \tfrac{3}{2}k-1)$.
\item   There exist integers $s,t$ with $st \le \frac{k}{2}$ and a cyclic system
        $\vec{\sigma}$ on $G$ so that $G$ is $(s,t)$-ring-like, and there exists an interval $J$ of $\vec{\sigma}$ so that the set $Q = \cup_{B \in J} B$ satisfies $A \subseteq Q$ and $|Q \setminus A| \le \tfrac{1}{2}k^3 + k^2$.
\end{enumerate}
\end{theorem}

This is a structure theorem giving a rough characterization of vertex transitive
graphs with small separations in the sense that any set $A$ which satisfies (i) or (ii) must have
$|\partial A|$ bounded as a function of $k$.  Indeed, if $A$ satisfies (i) then
$|\partial A| \le d |A| \le (2k^3+k^2) (\frac{3k}{2} - 1) \le 3k^4$
and if $A$ satisfies (ii) then $|\partial A| \le |\partial Q| + |Q \setminus A|
 \le 2st + \tfrac{1}{2} k^3 + k^2
 \le \tfrac{1}{2} k^3 + k^2 + k$.

Our theorem has an immediate consequence for separations in
Eulerian digraphs. Note that finite vertex transitive digraphs are
always Eulerian, so the difference only occurs in the infinite case.
Let $\vec{G}$ be a vertex transitive digraph, and let $G$ be the
underlying unoriented graph (which is clearly vertex transitive).
Let $A \subseteq V(\vec{G})$ be a finite vertex-set such that
$0 < |A| \le \tfrac{1}{2}|V(G)|$ and
%$G[A]$ is connected,
$G[A\cup\partial A]$ is connected, and set $k = |\partial^+(A)|$.
Let us also assume that $\diam(G) \ge 31(2k^2+1)^2$.
It follows from Theorem \ref{digraph-con_thm} that every vertex in $G$ has
indegree and outdegree $d$ where $d \le 2k-1$ so we have
$|\partial^-(A)| \le |\delta^-(A)| = |\delta^+(A)| \le k(2k-1)$ and we find that $|\partial A| \le 2k^2$
(in the unoriented graph $G$).  Thus, by the preceding theorem,
either $|A| \le 16k^6 +4k^4$ or
$G$ is $(s,t)$-ring-like with $st \le k^2$ and $A$ is efficiently contained in an interval.
\begin{corollary}
\label{cor:1.10}
Let $\vec{G}$ be a connected vertex transitive Eulerian digraph.
Let $A \subseteq V(\vec{G})$ be a finite vertex-set such that
$0 < |A| \le \tfrac{1}{2}|V(\vec{G})|$ and
$\vec{G}[A\cup\partial A]$ is connected, and set\/ $k = |\partial^+(A)|$.
Let us also assume that the diameter of the underlying undirected graph is at least $31(2k^2+1)^2$.
Then one of the following holds.
\begin{enumerate}[label={\rm (\roman{*}) \ }, ref={\rm (\arabic{*})}]
\item   $|A|\le 16k^6 +4k^4$.
\item   There exist integers $s,t$ with $st \le k^2$ and
a cyclic system $\vec{\sigma}$ on $\vec{G}$
so that $\vec{G}$ is $(s,t)$-ring-like and there exists
an interval $J$ of $\vec{\sigma}$ so that the set
$Q = \cup_{B \in J} B$ contains $A$ and
$|Q \setminus A| \le 4k^6 + 4k^4$.
\end{enumerate}
\end{corollary}

%\bojan{******* Omit *******}
%Later in the article we derive the following corollary for infinite Eulerian digraphs
%which has an outcome with a different structure (here $\depth^+(A)$ is
%the supremum over all $x \in A$ of the (digraph) distance from $x$ to $V(G) \setminus A$).
%
%\begin{corollary}
%\label{cor:1.10extended}
%Let $G$ be an infinite connected vertex transitive Eulerian digraph, let
%$A \subseteq V(G)$ be finite and set\/ $k = |\partial^+(A)|$.
%Then one of the following holds.
%\begin{enumerate}[label={\rm (\roman{*}) \ }, ref={\rm (\arabic{*})}]
%\item   $\depth^+(A) \le k$\/ and\/ $|A|\le *****$.
%\item   There exist integers $s,t$ with $st \le \frac{k}{2}$ and
%a cyclic system $\vec{\sigma}$ on $G$
%so that $G$ is $(s,t)$-ring-like and there exists
%an interval $J$ of $\vec{\sigma}$ so that the set
%$Q = \cup_{B \in J} B$ contains $A$ and
%$|Q \setminus A| \le \frac{1}{2}k^3 + k^2$.
%\end{enumerate}
%\end{corollary}
%
%The proof is given in the last section.
%\bojan{************************}

Interestingly, the same conclusion does not hold for (vertex transitive) digraphs which are not Eulerian.
Let $\vec{H}$ be an orientation of the infinite 3-regular tree such that every vertex has outdegree 1 and indegree 2. Then the vertex-set $B$ of a directed path has $|\partial^+(B)| = 1$ but $B$ may have
arbitrarily large size.

The main notion that we use in the proof is the depth of a set.  This is a convenient parameter for
our purposes, but leads us to make an assumption on the diameter of $G$ (to ``spread out" the graph) which is likely unnecessarily strong.  As far as we know, this theorem may be true without any such assumption.  Since we work primarily with depth, the bound on $\depth(A)$ in (i) is the natural consequence of our arguments.  To get a bound on the number of vertices in $A$ for (i) we (rather naively) apply the following pretty theorem which relates $|A|$, $|\partial A|$ and
$\diam(A)$.

\begin{theorem}[Babai and Szegedy \cite{BSz}]
\label{diam_bound_thm}
If $G$ is a connected vertex transitive graph and\/ $A \subset V(G)$ is a non-empty finite vertex-set with
$|A| \le \tfrac12 |V(G)|$, then
\[ \frac{| \partial A |}{|A|} \ge \frac{1}{ \diam(A) + 1}\,. \]
\end{theorem}

It appears likely that Theorem \ref{main_thm} should hold with a bound of the form $|A| \le ck^2$ instead of
$|A| \le 2k^3(1+o(1))$ in (i).  This strengthening would follow from the following conjecture that the diameter in Theorem \ref{diam_bound_thm} may be replaced by a constant multiple of the depth.

\begin{conjecture}
There exists a fixed constant $c>0$ so that in every connected vertex transitive graph we have
$\frac{| \partial A |}{ |A| } \ge \frac{c}{\depth(A)}$ whenever $A \subseteq V(G)$
is finite and $0 < |A| \le \tfrac12 |V(G)|$.
\end{conjecture}

\medskip

\subsection*{Asymptotic Expansion in Groups}

Asymptotic expansion or growth in groups is an extensive and well studied topic.
Here, instead of looking at $|AB|$ for a pair of finite sets $A,B$, we consider the asymptotic behavior of $|A^n|$ when $A$ is a generating set.  The major result in this area is the following theorem of Gromov which resolved (in the affirmative) a conjecture of Milnor.

\begin{theorem}[Gromov \cite{Gro}]
Let ${\mathcal G}$ be an infinite group, let $A \subseteq {\mathcal G}$ be a finite generating set, and assume further that $1 \in A$ and $\{ a^{-1} : a \in A \} = A$.  Then the function $n \mapsto |A^n|$ is bounded by a polynomial in $n$ if and only if ${\mathcal G}$ has a nilpotent subgroup of finite index.
\end{theorem}

In the special case that the growth is linear, the above theorem implies
that ${\mathcal G}$ has a subgroup isomorphic to $\ZZ$ of finite index,
and by a result of Freudenthal \cite{Fre} (see also Stallings \cite{Sta}),
this implies that
${\mathcal G}$ has a finite normal subgroup $N$ so that ${\mathcal G}/N$ is
either cyclic or dihedral. A clear proof of this special case,
which also features good explicit bounds, was obtained by Imrich and Seifter.

\begin{theorem}[Imrich and Seifter \cite{IS87}]
Let ${\mathcal G}$ be an infinite group, let $A \subseteq {\mathcal G}$
be a finite generating set, and assume
further that $1 \in A$ and $\{ a^{-1} : a \in A \} = A$.
If there exists an integer $k$ such that \/ $k \ge |A^k| - |A^{k-1}| =: q$,
then ${\mathcal G}$ has a cyclic subgroup of index $\le q$.
In particular, ${\mathcal G}$ has linear growth.
\end{theorem}

This result may also be obtained as a consequence of our Corollary \ref{cor:17}
which appears in the next section.

\medskip

\subsection*{Asymptotic Expansion in Graphs}

Before discussing this topic, we require two more definitions.
For any vertex $x \in V(G)$ and any positive integer $k$, we let $B(x,k)$ denote
the set of vertices at distance at most $k$ from $x$.  If $G$ is vertex transitive,
then $|B(x,k)| = |B(y,k)|$ for every $x,y \in V(G)$. The function
$b: \NN \to \NN$ given by $b(k) = |B(x,k)|$ is called the \DEF{growth function}
of~$G$.

The study of asymptotic expansion in graphs is the study of the behavior of
the growth function.  It is easy to see that if
$G = {\mathit Cay}({\mathcal G},A)$, where $1\in A$, then $b(k) = |A^k|$,
so this is a direct
generalization of the study of expansion in groups.  The following result is
the major accomplishment in this area and gives a direct generalization of
Gromov's theorem.

\begin{theorem}[Trofimov \cite{Tro}]
Let $G$ be a vertex transitive graph and assume that its growth function is bounded by a polynomial.
Then there exists a system of imprimitivity $\sigma$ with finite blocks so that $Aut(G^{\sigma})$ is
finitely generated, has a nilpotent subgroup of finite index, and the stabilizer of every vertex in $G^{\sigma}$ is finite.
\end{theorem}

As before, in the case when the growth function $b$ is bounded by a linear function,
the structure of $G$ can be obtained by a more elementary combinatorial argument,
as in the following result.

\begin{theorem}[Imrich and Seifter \cite{IS88}]
Let $G$ be an infinite connected vertex transitive graph, and let $b$ be
the growth function of $G$.  Then $G$ has two ends if and only if $b(n)$ is bounded by a linear function in $n$.
\end{theorem}

Our Theorem \ref{main_thm} can be used to obtain
a result similar to the above, but it also gives the following
explicit lower bound on the growth of infinite vertex
transitive graphs which are not ring-like.

\begin{corollary}
\label{cor:17}
If\/ $G$ is a connected infinite vertex transitive graph and a finite
vertex-set $A$ has $\depth(A) > |\partial A|$, then $G$ is
$(s,t)$-ring-like where $st \le \frac{1}{2} |\partial A|$.
In particular, $b(n) > \tfrac12 n(n+1)$ (for every $n\ge0$)
if\/ $G$ is not ring-like.
\end{corollary}

\begin{proof}
The first conclusion follows directly from Theorem \ref{main_thm}.
Consider now the set $A = B(x,n)$. Clearly, $\depth(A) \ge n+1$.
So, if the previous case does not apply, we conclude that
$|\partial A|\ge n+1$, and hence
$b(n+1) = |A| + |\partial A| \ge b(n) + n + 1$ for every $n\ge0$.
This implies that $b(n) \ge 1+\tfrac12 n(n+1)$ for $n\ge0$.
\end{proof}

\medskip

\subsection*{Structural Properties}

We now turn our attention away from expansion and toward the structure of vertex transitive graphs.
Next we state an important (yet unpublished) theorem of Babai which is related to our main theorem.

\begin{theorem}[Babai \cite{Bab}]
There exists a function $f:\NN\to\NN$ so that every finite vertex transitive
graph $G$ without $K_n$ as a minor satisfies one of the following properties:

\begin{enumerate}[label={\rm (\roman{*}) \ }, ref={\rm (\arabic{*})}]
\item $G$ is a vertex transitive map on the torus.
\item $G$ is $(f(n),f(n))$-ring-like.
\end{enumerate}
\end{theorem}

It appears likely to us that an inexplicit version of our theorem for finite graphs
might be obtained from Babai's theorem (which does not give the function $f$
explicitly).  However, at this time we do not have a proof of this.  Conversely,
our theorem can be used to obtain a strengthening of Babai's theorem with explicit
values for the function $f(n)$. We shall explore this in a subsequent paper \cite{DM}.
Here we only state the following corollary of Theorem \ref{main_thm}
that may be of independent interest. This result involves the notion of the tree-width,
whose definition is postponed until Section~\ref{sect:proofs}.

\begin{corollary}
\label{cor:19}
If\/ $G$ is a connected finite vertex transitive graph and $k$ is a positive integer, then
one of the following holds.
\begin{enumerate}[label={\rm (\roman{*}) \ }, ref={\rm (\arabic{*})}]
\item   $G$ is $(s,t)$-ring-like, where $2st\le k$.
\item   $G$ has tree-width $\ge k$.
\item   The degree of vertices in\/ $G$ is at most\/ $k - 1$ and the diameter of\/ $G$ is less than $31(k+1)^2$.
\end{enumerate}
\end{corollary}

The proof is given in the last section. It is easy to see that in the first case of Corollary \ref{cor:19}, the tree-width of $G$ is less than $k$. Let us observe that in the last case of
Corollary \ref{cor:19}, the degree of $G$ and the diameter of $G$ are both bounded in terms of $k$. Hence, the order of $G$ is bounded in terms of $k$, say $|V(G)|\le s(k)$. Consequently, $G$ is $(s(k),0)$-ring-like (and the tree-width of $G$ is less than $s(k)$).

The key tool we use to prove our main theorem is a structural lemma on
vertex transitive graphs which appears to be of independent interest.  Before stating
this lemma, we require another definition.  A finite subset $A \subseteq V(G)$
is called an $(s,t)$-\DEF{tube} if $G[A\cup \partial A]$
is connected, and there is a partition
of $\partial A$ into $\{L,R\}$ (with $L,R\ne\emptyset$) so that
$\dist_G(x,y) \le s$ whenever $x,y \in L$ or $x,y \in R$ and
$\dist_G(x,y) \ge t$ whenever $x \in L$ and $y \in R$.
Any partition satisfying this property is called a \DEF{boundary} partition.

\begin{lemma}[Tube Lemma]
\label{tube_lem}
Let $G$ be a vertex transitive graph.
If\/ $G$ has a $(k,3k+6)$-tube $A$ with boundary partition $\{L,R\}$ and
%$\depth(V(G) \setminus (A \cup \partial A)) \ge k+2$,
$\depth(V(G) \setminus A)) \ge k+1$,
then there exists a pair of integers $(s,t)$
and a cyclic system $\vec{\sigma}$ so that $G$ is $(s,t)$-ring-like
with respect to $\vec{\sigma}$, and $st \le \min\{|L|,|R|\}$.
\end{lemma}

The proof is postponed until Section \ref{sect:tube lemma}.

This lemma is also meaningful for groups (although the assumptions are more natural in the context of graphs): If $G$ is a Cayley graph for a group ${\mathcal G}$ and $G$ has a tube which satisfies the assumptions of the Tube Lemma, then $G$ is $(s,t)$-ring-like and it follows that ${\mathcal G}$ has a normal subgroup $N$ (of size $\le s$) so that ${\mathcal G}/N$ is either cyclic or dihedral.

\section{Uncrossing}

The main tool we use in the proofs of the Tube Lemma and our main theorem is
a simple uncrossing argument.  Indeed, this was the main tool
used to prove Theorems \ref{edge-con_thm}, \ref{vertex-con_thm}, and
\ref{digraph-con_thm} as well.  This argument is probably easiest to understand
with the help of a diagram, so we introduce one in Figure \ref{merge_fig}.
Here it is understood that $A_1,A_2$ are subsets of the vertex-set of a graph $G$, and the sets $P,Q,S,T,U,W,X,Y,Z$ are defined by the diagram. For example, $Q = \partial A_1 \cap A_2$,
$X = A_1\cap (V(G)\setminus (A_2\cup\partial A_2))$, etc.
For convenience, we will frequently refer back to this diagram.

\begin{figure}
\begin{center}
\includegraphics[width=60mm]{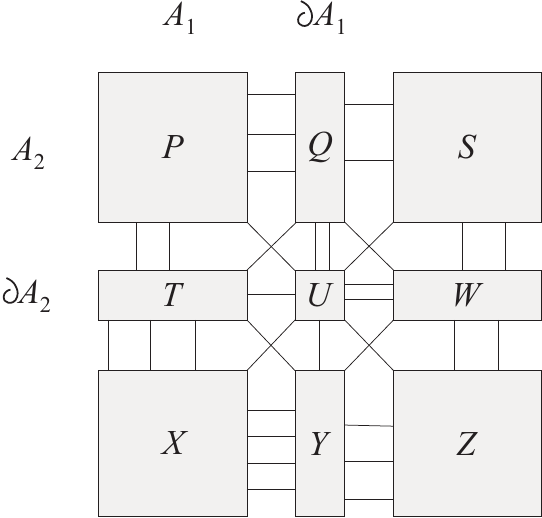}
\end{center}
\caption{The diagram for the uncrossing lemma. Edges between the sets are only possible where indicated.}
\label{merge_fig}
\end{figure}
%\begin{center}
%\begin{figure}[h]
%\begin{picture}(140,140)
%\put(68,5){\line(0,1){90}}
%\put(92,5){\line(0,1){90}}
%\put(35,38){\line(1,0){90}}
%\put(35,62){\line(1,0){90}}
%\put(47,77){$P$}
%\put(76,77){$Q$}
%\put(104,77){$S$}
%\put(47,47){$T$}
%\put(76,47){$U$}
%\put(104,47){$W$}
%\put(47,17){$X$}
%\put(76,17){$Y$}
%\put(104,17){$Z$}
%\put(44,114){$A_1$}
%\put(71,114){$\partial A_1$}
%\put(8,77){$A_2$}
%\put(4,47){$\partial A_2$}
%\end{picture}
%\caption{A Venn diagram for the uncrossing lemma}\label{merge_fig}
%\end{figure}
%\end{center}

\begin{lemma}[Uncrossing]
\label{uncrossing_lem}
Let $G$ be a graph, let $A_1,A_2 \subseteq V(G)$ and let the sets
$P$, $Q$, $S$, $T$, $U$, $W$, $X$, $Y$, $Z$ be defined as in Figure \ref{merge_fig}.
Then we have:
\begin{enumerate}[label={\rm (\roman{*}) }, ref={\rm (\arabic{*})}]
\item   %If $P,Z \neq \emptyset$ then
    $|\partial P| + |\partial (P \cup Q \cup S \cup T \cup X)|
    \le |\partial A_1| + |\partial A_2|$.
\smallskip
\item   %If $S,X \neq \emptyset$ then
    $|\partial S| + |\partial X| \le |\partial A_1| + |\partial A_2|$.
\smallskip
\item   If\/ $|\partial A_2| = |\partial P| = |\partial S| = k$, then
%        = |\partial X| = |\partial Z|$
    $|Q\cup U| = |\partial A_1 \cap (A_2 \cup \partial A_2)| \ge \frac{k}{2}$.
\end{enumerate}
\end{lemma}

\begin{proof}
Let us first observe that there are no edges from $P,Q,S$ to $X,Y,Z$ and
no edges from $P,T,X$ to $S,W,Z$. Therefore,
$\partial (P \cup Q \cup S \cup T \cup X) \subseteq U\cup W\cup Y$,
$\partial P \subseteq Q\cup U\cup T$, and
$\partial A_1 \cap \partial A_2 = U$. Then apply the inclusion-exclusion formula
to get (i).

For (ii), we similarly use the fact that $\partial S \subseteq Q \cup U \cup W$
and $\partial X \subseteq T \cup U \cup Y$.

To prove (iii), observe that $\partial A_1 \cap (A_2 \cup \partial A_2) = Q\cup U$.
Now,
$$
    |\partial P| + |\partial S| \le |T| + 2|Q\cup U| + |W| =
    |Q\cup U| + |\partial A_2| + |Q|.
$$
This implies that $|Q\cup U| + |Q| \ge k$, so $|Q\cup U|\ge \tfrac{k}{2}$.
\end{proof}

\section{Two-Ended Graphs}

The purpose of this section is to establish a theorem which gives us some detailed
structural information about vertex transitive graphs with two ends.  The main
tool we use is a corollary of an important theorem of Dunwoody.  However, since
Dunwoody's proof is rather tricky, and we have a proof of this corollary which
we consider to be more transparent, we have included it here.  This also has the
advantage of keeping the present article entirely self-contained.  Before stating
the main theorem from this section, we require some further definitions.

If $G$ is a graph, a \DEF{ray} in $G$ is a one-way-infinite path.
Two rays $r,s$ in a graph $G$
are \DEF{equivalent} if for any finite set of vertices $X$, the (unique) component
of $G \setminus X$ which contains infinitely many vertices of $r$ also
contains infinitely many vertices of
$s$.  This relation is immediately seen to be an equivalence relation, and the
corresponding equivalence classes are called the \DEF{ends} of the graph $G$.
By a theorem of Hopf \cite{Hop}
and Halin \cite{Hal}, every connected vertex transitive graph has either one, two, or infinitely many ends. We let
$$
   \kappa_{\infty}(G) =
      \inf \{|S| \,:\, S \subseteq V(G)
      \mbox{ and $G \setminus S$ has $\ge 2$ infinite components} \}.
$$
So $\kappa_{\infty}(G)$ is finite if and only if $G$ has at least two ends.

If $G$ is a graph which is ring-like with respect to the cyclic system $\vec{\sigma}$, then we say that $G$ is $q$-\DEF{cohesive} if any two vertices of $G$ which are in the same block of $\vec{\sigma}$ or in adjacent blocks of $\vec{\sigma}$ can be joined by a path of length at most $q$.  We are now ready to state the main result from this section.

\begin{theorem}
\label{two_ends_thm}
Let\/ $G$ be a connected vertex transitive graph with two ends.  Then there exist integers $s,t$ and a cyclic system $\vec{\sigma}$ so that $G$ is $(s,t)$-ring-like and $2st$-cohesive with respect to $\vec{\sigma}$, and $\kappa_{\infty}(G) = st$.
\end{theorem}

An important tool used to establish this result is Corollary \ref{dunwoody_cor}
below, which follows from the following strong result of Dunwoody.

\begin{theorem}[Dunwoody \cite{Dun}]
Let\/ $G$ be an infinite connected vertex transitive graph.
If there exists a finite edge-cut $\delta X$ of\/ $G$ so that both\/ $X$ and\/
$V(G) \setminus X$ are infinite, then there exists such an edge-cut $\delta Z$
with the additional property that for every automorphism $\phi$ of $G$
either $Z$ or $V(G) \setminus Z$ is included in
either $\phi(Z)$ or $\phi(V(G) \setminus Z)$.
\end{theorem}

\begin{corollary}[Dunwoody]
\label{dunwoody_cor}
If\/ $G$ is a connected vertex transitive graph with two ends, then there exists a cyclic system $\vec{\sigma}$ on $G$ with finite blocks.
\end{corollary}

We call a subset $X$ of vertices a \DEF{part} if both $X$ and
$V(G) \setminus X$ are infinite but $\partial X$ is finite.
If $X$ is a part, and $\epsilon$ is an end, then we say that $X$ \DEF{captures} $\epsilon$ if every ray in $\epsilon$ has all but finitely many vertices in $X$.
We call $X$ a \DEF{narrow} part if $|\partial X| = \kappa_{\infty}(G)$.

If $G$ is a vertex transitive graph with two ends, then every automorphism $\phi$ of $G$ either maps each end to itself, or interchanges the two ends.  We call automorphisms of the first type \DEF{shifts} and automorphisms of
the second type \DEF{reflections}.
Define a map ${\mathit sign} : Aut(G) \rightarrow \{-1,1\}$ by the rule that
${\mathit sign}(\phi) = 1$ if $\phi$ is a shift and ${\mathit sign}(\phi) = -1$ if $\phi$ is a reflection.
%It is immediate that ${\mathit sign}$ is a group homomorphism.

Now we are ready to provide a self-contained proof of Dunwoody's
Corollary \ref{dunwoody_cor}.

\begin{proof}[Proof of Corollary \ref{dunwoody_cor}]
Let us denote the ends of $G$ by ${\mathcal L}$ and ${\mathcal R}$.
It follows from the uncrossing lemma that
whenever $P,Q$ are narrow parts that capture ${\mathcal L}$, then
$P \cap Q$ and $P \cup Q$ are also narrow parts that also capture ${\mathcal L}$.
More generally, the set of narrow parts that capture ${\mathcal L}$ is
closed under finite intersections.  Note that by vertex transitivity every vertex is contained in a narrow part that captures ${\mathcal L}$ and a narrow part that captures ${\mathcal R}$.
For every $x \in V(G)$, let $L(x)$ $(R(x))$ be the intersection of all narrow parts which contain $x$ and capture ${\mathcal L}$ $({\mathcal R})$.
We claim that $L(x)$ is a narrow part. Clearly, $V(G)\setminus L(x)$ is infinite.
If $L(x)$ were finite, let $Y$ be the set of vertices at distance 1 or 2 from
$L(x)$. For each $y\in Y$, there is a narrow part $P(y)$ that contains $x$, captures
${\mathcal L}$, and does not contain $y$. Since $Y$ is finite, the intersection
$T = \cap_{y\in Y} P(y)$ is also a narrow part that
contains $L(x)$ but no other point at distance $\le 2$ from this set. Therefore,
$|\partial (T \setminus L(x))| = |\partial T|-|\partial L(x)| < |\partial T|$.
Since $T\setminus L(x)$ is also a part, the last inequality contradicts
the fact that $T$ is a narrow part. This shows that $L(x)$ is infinite.

Similarly, if $|\partial L(x) | > \kappa_{\infty}(G)$, then there exists a finite
set of narrow parts containing $x$ and capturing ${\mathcal L}$ with intersection
$T$ and $|\partial T| > \kappa_{\infty}(G)$,
a contradiction.  Thus, $L(x)$, and similarly $R(x)$, is a narrow part.

Next, define a map $\beta^R : V(G) \times V(G) \rightarrow \ZZ$
by the rule $\beta^R(x,y) = |R(x) \setminus R(y)| - |R(y) \setminus R(x)|$.
Let $x,y,z \in V(G)$, and define the following values:
\begin{eqnarray*}
a &=& |R(x) \setminus (R(y) \cup R(z))| \\
b &=& |(R(x) \cap R(y)) \setminus R(z)| \\
c &=& |R(y) \setminus (R(z) \cup R(x))| \\
d &=& |(R(y) \cap R(z)) \setminus R(x)| \\
e &=& |R(z) \setminus (R(y) \cup R(x))| \\
f &=& |(R(z) \cap R(x)) \setminus R(y)|
\end{eqnarray*}
Now we have that $\beta^R(x,y) + \beta^R(y,z)
    = (a + f) - (c + d) + (b + c) - (e + f) = (a + b) - (d + e) = \beta^R(x,z)$.
Define $\beta^L : V(G) \times V(G) \rightarrow \ZZ$ by the similar rule
$\beta^L(x,y) = |L(x) \setminus L(y)| - |L(y) \setminus L(x)|$ and observe that
$\beta^L(x,y) + \beta^L(y,z) = \beta^L(x,z)$ holds.  Next, define
$\beta : V(G) \times V(G) \rightarrow \ZZ$ by setting
$\beta(x,y) = \beta^R(x,y) - \beta^L(x,y)$ and note again that
$\beta(x,y) + \beta(y,z) = \beta(x,z)$.
If $\phi \in Aut(G)$ and $x \in V(G)$, then either ${\mathit sign}(\phi) = 1$,
$\phi(L(x)) = L(\phi(x))$, and $\phi(R(x)) = R(\phi(x))$, or ${\mathit sign}(\phi) = -1$,
$\phi(L(x)) = R(\phi(x))$, and $\phi(R(x)) = L(\phi(x))$.  It follows that
$\beta(x,y) = {\mathit sign}(\phi) \beta( \phi(x),\phi(y) )$ holds for
every $x,y \in V(G)$.

Now, define two vertices $x,y$ to be \DEF{equivalent} if $\beta(x,y) = 0$.
Note that this is an equivalence relation preserved by the automorphism group.
Let $\sigma$
be the corresponding system of imprimitivity.  If $B,B' \in \sigma$, then $\beta(x,x')$ has the same value for every $x \in B$ and $x' \in B'$ and we define $\beta(B,B')$ to be this value.  Next, define a relation on
$\sigma$ as follows.  For any block $B \in \sigma$, there is a unique block $B'$ for which
$\beta(B,B')$ is minimally positive.
%and a unique block $B''$ so that $\beta(B,B'')$ is maximal subject to being negative.
Include $(B,B')$
%and $(B,B'')$
in our relation.  It follows immediately that this relation imposes a cyclic order which is preserved by any automorphism of the graph, so we have a cyclic system $\vec{\sigma}$ as desired.

It remains to show that the blocks of $\sigma$ are finite. Let
$$
   F(x) = \partial L(x)\cup \partial R(x)\cup
          (V(G)\setminus(\partial L(x)\cup \partial R(x))).
$$
Observe that $F(x)$ is finite. Set $h = \max\{\dist(x,z)\mid z\in F(x)\}$.
Note that $h$ is independent of $x\in V(G)$ since $G$ is vertex transitive and since
every automorphism $\phi$ of $G$ maps $L(x)$ and $R(x)$ onto $L(\phi(x))$ and
$R(\phi(x))$ (possibly interchanging $L$ and $R$).
If $\dist(x,y) > 2h$, then it is easy to see that either $y\in L(x)$
and $x\in R(y)$, or $y\in R(x)$ and $x\in L(y)$. Assuming the former,
$L(x)$ is a narrow part containing $y$ and capturing ${\mathcal L}$, so
$L(y)\subseteq L(x)$. Since $\dist(x,y) > 2h$, we have
$\dist(\partial L(x),y) > h$, hence $\partial L(y)\subseteq L(x)$.
Thus
$$
   \beta^L(x,y) = |L(x)\setminus L(y)| \ge |\partial L(y)| > 0.
$$
Similarly, since $x\in R(y)$, we conclude that
$R(x)\cup \partial R(x)\subseteq R(y)$, and consequently
$$
   \beta^R(x,y) = - |R(y)\setminus R(x)| \le - |\partial R(x)| < 0.
$$
Finally, this yields that $\beta(x,y)\ne 0$ and shows that the blocks of
$\sigma$ are finite. This completes the proof.
\end{proof}

If $G$ is a vertex transitive graph with two ends, then we may define a relation $\sim$ on
$V(G)$ by the rule $x \sim y$ if there exists a shift $\phi \in Aut(G)$ with $\phi(x) = y$.  It is immediate from the definitions that $\sim$ is an equivalence relation preserved by $Aut(G)$, and we let $\tau$ denote the corresponding system of imprimitivity.  Since the product of two reflections is a shift, $|\tau| \le 2$.  We define $G$ to be \DEF{Type $i$} if $|\tau| = i$.  Graphs of Type 1 will be easiest to work with, since in this case we have shifts taking any vertex to any other vertex.  If $G$ is a graph of Type 2, then we view $\tau$ as
a (not necessarily proper) 2-coloring of the vertices.
In this case, every shift fixes both color classes, and every
reflection interchanges them.
An example of a Type 2 graph is illustrated in Figure~\ref{fig:2}.

%\begin{figure}
%\insertepsfile{symcut1.eps}
%\label{fig:2}
%\caption{A vertex transitive graph of Type 2}
%\end{figure}
\begin{figure}
\begin{center}
\includegraphics[width=80mm]{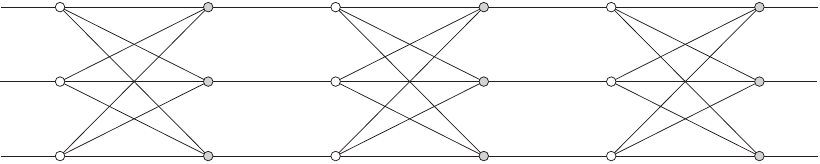}
\end{center}
\caption{A vertex transitive graph of Type 2}
\label{fig:2}
\end{figure}

If $X,Y$ are disjoint subsets of $V(G)$, we say that $X$ and $Y$ are
\DEF{neighborly} if every point in $X$ has a neighbor in $Y$ and every point in $Y$ has a neighbor in $X$.  We say that $G$ is \DEF{tightly $(s,t)$-ring-like} with respect to $\vec{\sigma}$ if $G$ is $(s,t)$-ring-like with respect to
$\vec{\sigma}$, and further, every pair of blocks in $\vec{\sigma}$ at distance $t$ are neighborly.

\begin{lemma}
\label{tight_lem}
If\/ $G$ is a connected vertex transitive graph with two ends, then there exist integers $s,t$, and a cyclic system $\vec{\sigma}$ so that $G$ is tightly
$(s,t)$-ring-like with respect to $\vec{\sigma}$.
\end{lemma}

\begin{proof}
By Corollary \ref{dunwoody_cor} we may choose a cyclic system $\vec{\sigma}$ where
$\sigma = \{ X_i : i \in \ZZ\}$ and the cyclic order is $\ldots, X_{-1},X_0,X_1,\ldots$.
Set $s$ to be the size of a block of $\sigma$ and $t$ to be the largest integer so that there exist adjacent vertices which lie in blocks at distance $t$.  Then, $G$ is $(s,t)$-ring-like with respect to~$\vec{\sigma}$.

First suppose that $G$ is Type 1 and choose $i \in \ZZ$ so that
$E[X_i,X_{i+t}] \neq \emptyset$.  Then every point in $X_i$ must have a neighbor in $X_{i+t}$ since there exists a shift taking any point in $X_i$ to any other point in this block, and such a map must fix $X_{t+i}$.  Similarly, every point in $X_{i+t}$ has a neighbor in $X_i$, so $X_i$ and $X_{i+t}$ are neighborly.
For every $j$, there exists a shift which sends $X_i$ to $X_j$, so $X_j$ and $X_{j+t}$ are also neighborly.  It follows from this that $G$ is tightly $(s,t)$-ring-like with respect to~$\sigma$.

Thus, we may assume that $G$ is of Type 2, and we let $\tau = \{Y_1,Y_2\}$ be the corresponding system of imprimitivity.
If $\sigma$ is not a refinement of $\tau$, then
$\{ X_i \cap Y_j : \mbox{$i \in \ZZ$ and $j \in \{1,2\}$}\}$
is a system of imprimitivity.
For $Z_{2j} = X_j \cap Y_1$, $Z_{2j+1} = X_j \cap Y_2$ ($j\in\ZZ$),
the cyclic ordering $\dots,Z_{-2},Z_{-1},Z_0,Z_1,Z_2,\dots$
is preserved by $Aut(G)$.  Thus, by possibly adjusting $\vec{\sigma}$ and $s$ and $t$, we may assume that $\sigma$ is a refinement of $\tau$.  In particular, for every $x,y \in X_i$ there exists an automorphism sending $x$ to $y$ which fixes every block of $\vec{\sigma}$ (since every shift sending $x$ to $y$ has this property).  So $X_i$ and $X_j$ are neighborly whenever $E[X_i,X_j] \neq \emptyset$.

Note that we may modify the cyclic order on $\sigma$ by ``shifting the even blocks $2k$ steps to the right", replacing $X_{2i}$ by $X_{2i-2k}$ for every $i \in \ZZ$ to obtain a new cyclic ordering which is preserved by $Aut(G)$.
Set $t_0 = \sup \{ i \in 2\ZZ : E[X_0,X_i] \neq \emptyset \}$ (if such $i \in 2\ZZ$
does not exist, let $t_0=0$), set $t^-_1 = \min \{ j \in 2\ZZ + 1 : E[X_0,X_j] \neq \emptyset \}$ and set $t^+_1 = \max \{j \in 2\ZZ + 1 : E[X_0,X_j] \neq \emptyset \}$.  Since there must be a block with odd index joined to $X_0$, $t_1^+$ and $t_1^-$ both exist. By shifting even blocks, we may further assume that either $t^+_1 = - t^-_1$ or $t_1^+ = 2 - t_1^-$.  If
$t_0 \ge t_1^+$, then $t=t_0$ and $E[X_i,X_{i+t}] \neq \emptyset$ for every $i \in \ZZ$ so $X_i$ and $X_{i+t}$ are neighborly for every $i \in \ZZ$ and we are done.  Similarly, if $t_1^+ = -t_1^- > t_0$, then $t = t_1^+$ and $X_i$ and $X_{i+t}$ are neighborly for every $i \in \ZZ$ and we are done.  The only remaining possibility is that
$t = t_1^+ = 2 - t_1^- > t_0$.  In this case, set $\sigma' = \{ X_{2i} \cup X_{2i+1} : i \in \ZZ \}$.  Then
$\sigma'$ is a system of imprimitivity, $\ldots, X_{-2} \cup X_{-1}, X_0 \cup X_1, X_2 \cup X_3, \ldots$ is  a cyclic order preserved by $Aut(G)$, and by construction, $G$ is tightly
$(2s,\frac{t-1}{2})$-ring-like with respect to $\sigma'$ and this ordering.
\end{proof}

\begin{lemma}
\label{connectivity_lem}
If\/ $G$ is an infinite connected vertex transitive graph that is tightly
$(s,t)$-ring-like, then $\kappa_{\infty}(G) = st$.
\end{lemma}

\begin{proof}
It is immediate that $\kappa_{\infty}(G) \le st$ as the removal of $t$ consecutive blocks of size $s$ leaves a graph with at least two infinite components.  Thus, it suffices to show that $st \le \kappa_{\infty}(G)$.

Assume that $G$ is tightly $(s,t)$-ring-like with respect to $\vec{\sigma}$
where $\sigma = \{X_i : i \in \ZZ\}$ and the cyclic order is given by $\ldots,X_{-1},X_0,X_1,\ldots$.  Next, choose $A \subseteq V(G)$ with $A$ and $V(G) \setminus A$ infinite so that

(i) $|\partial A|$ is minimum;

(ii) $T = \{ y \in V(G) \setminus A : \sigma_y \cap A \neq \emptyset \} $ is minimal subject to (i), where $\sigma_y$ denotes the block of $\sigma$ containing $y$.

It follows from our assumptions that $|\partial A| = \kappa_{\infty}(G)$ is finite.
Further, since there is a fixed upper bound on the maximum distance between two vertices
in the same block of $\sigma$, the set $T$ is finite.  Suppose (for a contradiction)
that there exist points $x,y$ in the same block of $\sigma$ and a shift $\phi \in Aut(G)$
so that $x \in A$, $y \not\in A$, and so that $\phi(x) = y$.  Then $\phi$ must fix every
block of $\sigma$, so the symmetric difference of $A$ and $\phi(A)$ is finite.  By
uncrossing (Lemma \ref{uncrossing_lem}), we have
$|\partial (A \cap \phi(A))| + |\partial (A \cup \phi(A))| \le 2|\partial A|$.  But then
it follows from (i) that $|\partial (A \cup \phi(A))| = |\partial A|$ and we see that
$A \cup \phi(A)$ contradicts our choice of $A$ for (ii).  Thus, no such $x,y,\phi$ can exist.

Now suppose that there are shifts taking any point in a block of $\sigma$ to any other point in this block.  It then follows from the above argument that both $A$ and
$\partial A$ are unions of blocks of $\sigma$.
Let $Q_k = \cup_{i \in \ZZZ} X_{it + k}$ for every $0 \le k \le t-1$.
Then $Q_k \cap \partial A$ must include a block of $\sigma$
for every $0 \le k \le t-1$ so $\kappa_{\infty}(G) = |\partial A| \ge st$ as desired.

Thus, we may assume that there exist $x,y \in X_0$ so that no shift maps $x$ to $y$.
So, $G$ is Type 2, and setting $\tau = \{Y_1,Y_2\}$ to be the corresponding 2-coloring,
we find that $\sigma' = \{ X_i \cap Y_j : \mbox{$i \in \ZZ$ and $j \in \{1,2\}$ } \}$
is a proper refinement of $\sigma$ and $\tau$.  Furthermore, it follows from our earlier
analysis that both $A$ and $\partial A$ are unions of blocks of $\sigma'$.  It follows
from the assumption that $G$ is tightly $(s,t)$-ring-like that either
$X_0 \cap Y_j$ and $X_t \cap Y_j$ are neighborly for $j=1,2$ or that
$X_0 \cap Y_j$ and $X_t \cap Y_l$ are neighborly whenever $\{j,l\} = \{1,2\}$.  In the
former case, setting $Q_k^j = \cup_{i \in \ZZZ} (X_{it + k} \cap Y_j)$ for $0 \le k \le t-1$
and $j = 1,2$ we find that $\partial A \cap Q_k^j$ includes a block of $\sigma'$ for
every $0 \le k \le t-1$ and $j=1,2$ so $\kappa_{\infty}(G) = |\partial A| \ge st$
as desired.  In the latter case, setting
$Q_k^j = \cup_{i \in \ZZZ} (X_{it+k} \cap Y_{(j+i)\pmod 2})$ for $0 \le k \le t-1$ and
$j = 1,2$ we find that $\partial A \cap Q_k^j$ includes a block of $\sigma'$ for
every $0 \le k \le t-1$ and $j = 1,2$.
Again, this implies that $\kappa_{\infty}(G) = |\partial A| \ge st$
and we are finished.
\end{proof}

In the next proof we will need a simple lema about short paths in finite graphs.

\begin{lemma}
\label{lem:cycles}
Let\/ $H$ be a finite connected graph, possibly with multiple edges, and let
$Q_1,\dots,Q_h$ be pairwise disjoint cycles in $H$. For every $x,y\in V(H)$,
there exists an $(x,y)$-path $P$ in $H$ such that for every $i$ $(1\le i\le h)$
the intersection $P\cap Q_i$ is either empty or a segment of $Q_i$
containing at most half of the edges of $Q_i$.
\end{lemma}

\begin{proof}
Let $P$ be an $(x,y)$-path. We say that the cycle $Q_i$ is
\DEF{badly traversed} by $P$ if $P\cap Q_i$ is not as claimed.
Suppose now that $Q_i$ is badly traversed, and let $U=V(P\cap Q_i)$.
Let $a,b$ be the first and the last vertex, respectively, on $P$ that belongs
to $Q_i$. If we replace the $(a,b)$-segment of $P$ by a shortest
$(a,b)$-segment on $Q_i$, we get another $(x,y)$-path $P'$. Clearly, $P'$ does
not introduce any new badly traversed cycles among $Q_1,\dots,Q_h$, and
repairs bad traversing of $Q_i$. This procedure thus leads to an $(x,y)$-path
which is as claimed.
\end{proof}

\begin{lemma}
\label{cohesive_lem}
If\/ $G$ is an infinite connected vertex transitive graph that is
tightly $(s,t)$-ring-like with respect to $\vec{\sigma}$, then it is
$2st$-cohesive with respect to $\vec{\sigma}$.
\end{lemma}

\begin{proof}
Assume that $\sigma$ has blocks $\{X_i : i \in \ZZ\}$ and the cyclic order in $\vec{\sigma}$ is given by $\ldots,X_{-1},X_0,X_1,\ldots$.
Let $x_0\in X_0$ and $y_0\in X_{-1}\cup X_0\cup X_1$.
Our goal is to prove that $\dist_G(x_0,y_0) \le 2st$.

Let $x_1\in X_t$ be a neighbor of $x_0$, and let $x_2\in X_{2t}$ be a neighbor
of $x_1$. Since shifts of $G$ have only one or two orbits, it is possible to
chose $x_2$ such that $x_0$ and $x_2$ are in the same orbit, i.e., there is a
shift $\alpha$ mapping $x_0$ to $x_2$. For $i\in\ZZ$, let $x_{2i}=\alpha^i(x_0)$
and let $x_{2i+1}=\alpha^i(x_1)$. Note that vertices
$\dots,x_{-2},x_{-1},x_0,x_1,x_2,x_3,\dots$ form a two-way-infinite path $P_0$
in $G$ that is preserved under the action of $\alpha$.

Let $H$ be the quotient graph of $G$ under the action of $\alpha$. More precisely,
vertices of $H$ are the orbits of the action of $\alpha$ on $V(G)$, and two of them
are adjacent if there is an edge in $G$ joining the two orbits. Since $\alpha$ is
an automorphism of $G$ such that $\alpha^i$ has no fixed points
if $i\in \ZZ\setminus\{0\}$, $G$ is a cover of $H$.
Let $x$ and $y$ be the vertices of $H$ that are orbits of $x_0$ and $y_0$,
respectively. Since $G$ is connected, there is a path $Q$ from $y$ to $x$ in $H$.
Since $|V(H)|=2st$, it is immediate that $|V(Q)|\le 2st$.

By the unique lifting property of paths in covering spaces, $Q$ can be lifted to
a path $\tilde Q$ in $G$ joining $y_0$ and some vertex $x_0'$ in the same
$\alpha$-orbit as $x_0$. Note that $x_0'\in X_{2mt}$ ($m\in \ZZ$),
so $x_0$ can be reached from $x_0'$ by using $2|m|$ edges on the path $P_0$.
This shows that $\dist_G(x_0,y_0) \le |V(Q)| - 1 + 2|m|$.

In order to show that the last inequality above implies that $G$ is $2st$-cohesive,
we need to show that $|V(Q)|+2|m|-1 \le 2st$.
We shall use Lemma \ref{lem:cycles}, so we first define cycles $Q_j$ covering
all vertices in $H$. The edges between $X_i$ and $X_{i+t}$ form a bipartite graph.
Each component of this bipartite graph is a regular bipartite graph of positive
degree since $G$ is neighborly $(s,t)$-ring-like. By K\"onig's theorem,
there is a perfect matching $M_i$ between $X_i$ and $X_{i+t}$. We choose such
perfect matchings arbitrarily for all $i=0,1,\dots,2t-1$. The projection of all
matchings $M_0,M_1,\dots,M_{2t-1}$ into $H$ gives rise to a collection of disjoint
cycles $Q_j$ ($j\in J$) covering all vertices of $H$. Also note that all these
cycles are of even length (possibly length 2). Let us now assume that
the $(x,y)$-path $Q$ satisfies the conclusion of Lemma \ref{lem:cycles}.
If $Q$ contains half of the edges of some cycle $Q_j$, in which case we will say
that the cycle $Q_j$ is \DEF{problematic}, then we have the freedom
to choose one or the other segment of $Q_j$ to be included in $Q$. Note that
the definition of the cycles $Q_j$ implies the following property: if $v\in V(Q_j)$
and $vv_1,vv_2$ are the two edges on $Q_j$ incident with $v$, then for every
lift of the path $v_1 v v_2$ to a path $\tilde v_1 \tilde v \tilde v_2$ in $G$,
if $\tilde v\in X_i$, then one of the vertices $\tilde v_1, \tilde v_2$ is in
$X_{i-t}$ and the other one is in $X_{i+t}$. This property and freedom to choose
either half of problematic cycles $Q_j$ enables us to assume the
following. Let $Q=u_0u_1u_2\dots u_r$ (where $u_0=y$ and $u_r=x$) and let
$\tilde Q = \tilde u_0\tilde u_1 \tilde u_2\dots \tilde u_r$.

\begin{itemize}
\item[(*)] If a cycle $Q_j$ ($j\in J$) sucks and if $u_i$ ($i\ge1$) is the first
vertex of $Q$ on $Q_j$, then the $\vec\sigma$-distance of $\tilde u_{i-1}$ and
$\tilde u_{i+1}$ is at most $t$. Similarly, if the cycle containing $u_0$ sucks,
then $\tilde u_1\in X_l$, where $|l|\le t$.
\end{itemize}

Now, let $k$ be the number of problematic cycles $Q_j$. Then
$|V(Q)|\le \tfrac12 (|V(H)|+k) = st + \tfrac{k}{2}$.
Property (*) guarantees that for every problematic cycle $Q_j$ (except possibly the first one) we save one for the backtracking on $P_0$, i.e. $2|m|\le |V(Q)|-1-(k-1)$.
Therefore
$$
\dist_G(x_0,y_0) \le |V(Q)| - 1 + 2|m| \le 2|V(\tilde Q)|-k-1
                 \le 2st - 1.
$$
This completes the proof.
\end{proof}

\begin{proof}[Proof of Theorem \ref{two_ends_thm}]
This is an immediate consequence of Lemmas \ref{tight_lem},
\ref{connectivity_lem}, and \ref{cohesive_lem}.
\end{proof}

\section{The Tube Lemma}

\label{sect:tube lemma}

The goal of this section is to prove (a slight strengthening of)
the Tube Lemma \ref{tube_lem}.  We begin by proving a lemma similar in spirit, but for vertex cuts where all points in the cut are close to the same vertex.

\begin{lemma}
\label{balloon_lem}
Let\/ $X \subseteq V(G)$ be a finite vertex-set in a vertex transitive
graph $G$ and assume that there exists $y \in V(G)$ so that
$\dist_G(y,x) \le k$ for every $x \in X$.
If there is a finite component $H$ of
$G \setminus X$ with $\depth(V(H)) \ge k+2$, then $\depth(V(H')) < k + 2$  for every other component $H'$ of $G \setminus X$.
\end{lemma}

\begin{proof}
Let $X$ be a minimal counterexample to the lemma. Choose a finite component of
$G \setminus X$ with vertex-set $A$ so that $\depth(A) \ge k + 2$,
and let $B$ be the vertex-set of another (finite or infinite) component
of $G\setminus X$ with $\depth(B) \ge k + 2$. We may assume that
$|A| \le |B|$.  Next choose a point $z \in A$ with depth $\ge k+2$ and choose an automorphism $\phi$ with
$\phi(y) = z$.  Since $A$ contains the ball of radius $k+1$ around $z$, we have
$\phi(X) \subseteq A$ and $A \setminus (\phi(A) \cup \phi(X)) \neq \emptyset$.
It follows from this that $S = \phi(A) \setminus (A \cup X) \neq \emptyset$.  Furthermore, by our assumption on $|A|$, we have $T = V(G) \setminus (A \cup X \cup \phi(A) \cup \phi(X)) \neq \emptyset$.  Now, set $X' = X \cap \phi(A)$ and $X'' = X \setminus \phi(A)$.
Then $X'$, $X''$ are vertex cuts
which separate $S$ and $T$ (respectively) from the rest of the vertices.
Since $S\subset \phi(A)$ and $A$ is finite, we have $|S| < |A| \le |B|$.
This implies that $B\subseteq T$ is a component of $G\setminus X''$ of depth
$\ge k+2$. Another component of $G\setminus X''$ has vertex-set
$A\cup \phi(A)$. It is finite and has depth $\ge k+2$ as well.
Therefore, $X''$ contradicts our choice of $X$ as a minimum counterexample.
\end{proof}

For $i=1,2$ let $A_i$ be a tube with boundary partition $\{L_i,R_i\}$ and let
$P,Q,S,T,U,W,X,Y,Z$ be the sets indicated by the diagram in Figure~\ref{merge_fig}.
We say that $A_1$ and $A_2$ \DEF{merge} if $P,S,X,Z \neq \emptyset$, $\{Q,Y\} = \{L_1,R_1\}$ and $\{T,W\} = \{L_2,R_2\}$. Note that these conditions imply that $U=\emptyset$; see Figure \ref{fig:merge tubes} for intuition.

\begin{figure}
\begin{center}
\includegraphics[width=80mm]{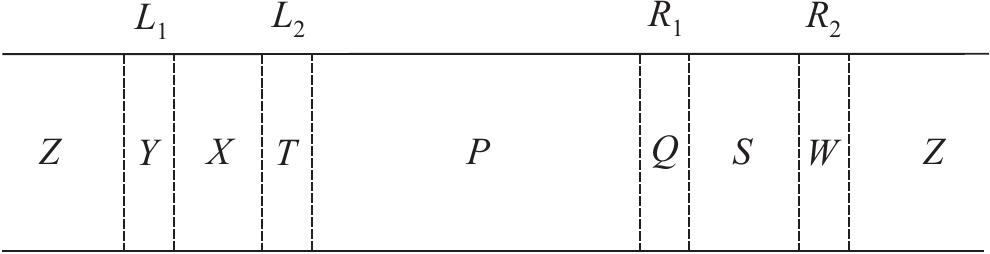}
\end{center}
\caption{Merging two overlapping tubes into a bigger one.}
\label{fig:merge tubes}
\end{figure}

The following lemma will be used to guarantee that tubes merge.

\begin{lemma}
\label{merge_lem}
Let $G$ be a connected vertex transitive graph, and for $i=1,2$ let $A_i$ be a $(k,k+2)$-tube in $G$ with boundary partition $\{L_i,R_i\}$ and
$\depth(V(G) \setminus (A_i \cup \partial A_i)) \ge k+2$.  Let
$P,Q,S,T,U,W,X,Y,Z$ be the sets indicated in Figure \ref{merge_fig}.
If $P,S,X,Z \neq \emptyset$ and
$\dist(\partial A_1, \partial A_2) \ge \frac{k+1}{2}$, then $A_1$ and $A_2$ merge.
\end{lemma}

\begin{proof}
It follows from the assumption $\dist(\partial A_1,\partial A_2) \ge \frac{k+1}{2}$ that
$U = \emptyset$.  The sets $T$ and $Q$ cannot be empty since $G[A_1\cup \partial A_1]$ and $G[A_2\cup \partial A_2]$ are connected and $P,S,X \neq \emptyset$.
Suppose (for a contradiction) that $W = \emptyset$.  Then
$Y \neq \emptyset$ since $Z \neq \emptyset$ and $G$ is connected.  Furthermore
$\dist(Q,Y) \ge \dist(Q, \partial A_2) + \dist(\partial A_2,Y) \ge k+1$ and it follows that
$\{Q,Y\} = \{L_1,R_1\}$.  But then applying Lemma \ref{balloon_lem} to either $L_1$ or $R_1$ gives us a contradiction.  Thus $W \neq \emptyset$ and similarly $Y \neq \emptyset$.
Again, we have $\dist(Q,Y) \ge \dist(Q,\partial A_2) + \dist(\partial A_2,Y) \ge k+1$ and similarly $\dist(T,W) \ge \dist(T, \partial A_1) +
\dist(\partial A_1,W) \allowbreak \ge k+1$. It follows that $\{Q,Y\} = \{L_1,R_1\}$ and $\{T,W\} = \{L_2, R_2 \}$ as required.
\end{proof}

The key ingredient in the proof of our tube lemma for finite graphs is the construction of a certain graph cover.  We then use this cover together with Corollary \ref{dunwoody_cor} to obtain the desired structure.  Our construction is based on voltage assignments, and the reader is referred to
Gross and Tucker \cite{GT} for a good introduction to this area.
Some further notation and definitions will be introduced in the proof of the
Tube Lemma.

\begin{lemma}
Let\/ $G$ be a vertex transitive graph.
If\/ $G$ has a $(k,3k+6)$-tube $A$ with boundary partition $\{L,R\}$ and
$\depth(V(G) \setminus (A \cup \partial A)) \ge k+2$, then there exists a pair of integers $(s,t)$
and a cyclic system $\vec{\sigma}$ so that\/ $G$ is $(s,t)$-ring-like
and\/ $2st$-cohesive with
respect to $\vec{\sigma}$, and $st \le \min\{|L|,|R|\}$.
\end{lemma}

\begin{proof}
Choose a $(k,3k+6)$-tube $A$ with $\depth(V(G) \setminus (A \cup \partial A)) \ge k+2$ and
boundary partition $\{L,R\}$ so that:

~(i) $\min\{ |L|, |R| \}$ is as small as possible,

\,(ii) $|L| + |R|$ is as small as possible subject to (i),

(iii) $|A|$ is as small as possible subject to (i) and (ii).

\noindent{It suffices to show that $G$ is $(s,t)$-ring-like where $st \le \min\{|L|,|R|\}$.
The proof consists of a series of four claims numbered (0)--(3), followed by a split into two cases, depending on wether $G$ is finite or infinite.}

\medskip

\noindent{\bf (0)} Every vertex in $\partial A$ has a neighbor in
$V(G) \setminus (A \cup \partial A)$.

\smallskip

Suppose (for a contradiction) that $x \in \partial A$ has no such neighbor.  If
$\{x\} \neq L$ and $\{x\} \neq R$ then $A \cup \{x\}$ contradicts our choice of $A$
for (i) or (ii).  But if $\{x\} = L$ or $\{x\} = R$, then $\partial (A \cup \{x\})$ is included in
either $R$ or $L$, and applying Lemma \ref{balloon_lem}
to this set gives a contradiction.

\bigskip

\noindent{\bf (1)} If $G$ is finite, then $G \setminus E[A,L]$ and $G \setminus E[A,R]$ are connected, and
$|A| \le |V(G) \setminus (A \cup \partial A)|$.

\smallskip

Suppose that $G \setminus E[A,L]$ is not connected, let $B$ be the vertex-set of a
component of this graph with $B \cap A = \emptyset$ and set $L' = B \cap L$ and
$B' = B \setminus L$.  It is an immediate consequence of Lemma \ref{balloon_lem}
(applied to $L'$) that $\depth(B') \le k+1$.  It then follows from the same
lemma (applied to $R$) that $L' \neq L$.  But then, $A \cup B$ is a tube which
contradicts our choice of $A$ for (i) or (ii).  Thus $G \setminus E[A,L]$ is connected,
and by a similar argument $G \setminus E[A,R]$ is also connected.  It follows from
this that there is a component of $G \setminus (A \cup L \cup R)$ with vertex-set $C$
and $L \cap \partial C \neq \emptyset \neq R \cap \partial C$.
Since $G$ is finite, $C$ is a $(k,3k+6)$-tube and by assumption (iii),
we conclude that $|A| \le |C|$, so $|A| \le |V(G) \setminus (A \cup \partial A)|$ as desired.

\bigskip

\noindent{\bf (2)} If $\phi \in Aut(G)$ satisfies $A \cap \phi(A) \neq \emptyset$ and
$\dist(\partial A,\partial \phi(A)) \ge \frac{k+1}{2}$, then $A$ and
$\phi(A)$ merge.

\smallskip

Set $A_1 = A$ and $A_2 = \phi(A)$ and define the sets $P,Q,S,T,U,W,X,Y,Z$ as in Figure
\ref{merge_fig}.  It follows immediately from our assumptions that $U = \emptyset$ and that
$P,T \cup X, Q \cup S \neq \emptyset$.  If $S = \emptyset$, then
$Q \neq \emptyset$, but every point in this set has all its neighbors in $P \cup Q = A \cup \partial A$
and we have a contradiction to claim (0).  Thus $S \neq \emptyset$ and similarly $X \neq \emptyset$.
If $Z = \emptyset$ then $G$ is finite, and by (0) we have
$W = Y = \emptyset$. But then
$|V(G) \setminus (A \cup \partial A)| = |S| < |P| + |Q| + |S| = |A_2| = |A|$
contradicting (1).  Thus $Z \neq \emptyset$.  Now Lemma \ref{merge_lem} shows that
$A$ and $\phi(A)$ merge, as claimed.

\bigskip

Choose a shortest path $D$ in $G[A \cup \partial A]$ from $L$ to $R$.  Let $v_{-2}$ be the
end of $D$ in $L$, let $v_2$ be the end of $D$ in $R$, and let $r \in \ZZ$
be such that the
length of $D$ is either $2r$ or $2r+1$.  Note that by our assumptions, $r \ge \lceil \frac{3k+5}{2} \rceil$.  Choose a vertex $v_0$ in
``the middle of'' $D$, i.e., $r \le \dist(v_0,v_i) \le r+1$ for $i\in\{-2,2\}$, and
choose vertices $v_{-1},v_1 \in V(D)$ at distance $\lceil \frac{3k+5}{2} \rceil$ from $v_0$ so that
$v_{-1}$ lies on the subpath of $D$ from $v_0$ to $v_{-2}$ and $v_1$ lies on the subpath from $v_0$ to $v_2$.

\bigskip

\noindent{\bf (3)} Let $\phi$ be an automorphism of $G$ and assume that either
$\dist(\phi(v_0),v_{-1}) \le 1$ or
$\dist(\phi(v_0),v_{1}) \le 1$.  Then
$\dist(\partial A, \partial \phi(A)) \ge \frac{k+1}{2}$, and $A$ and $\phi(A)$ merge.

\smallskip

We give the proof in the case $\dist(\phi(v_0),v_{-1}) \le 1$; the other case follows by a similar argument.  Let $y \in L$ and $x \in \partial \phi(A)$.  Then we have
\begin{eqnarray*}
\dist(y, x)
    &\ge&   \dist(\phi(v_0),x) - \dist(\phi(v_0),v_{-2})
        - \dist(v_{-2},y)  \\
    &\ge&   r - (r - \tfrac{3k+5}{2} + 2) - k   \\
    &=& \tfrac{k+1}{2}.
\end{eqnarray*}
Next let $y \in R$.  Then any path from $\phi(v_0)$ to $y$ which does not contain a point in
$L$ has length $\ge r + \frac{3k+5}{2} - 1 = r + \frac{3k+3}{2}$ and any such path which does contain a point in $L$ has length $\ge \dist(\phi(v_0),L) + \dist(L,R)
\ge (r - \frac{3k+6}{2} - 1) + (3k+6) \ge r + \frac{3k+4}{2}$.  It follows that
$\dist(\phi(v_0),y) \ge r + \frac{3k+3}{2}$.  Thus, for every $x \in \partial \phi(A)$ we have
\begin{eqnarray*}
\dist(x,y)
    &\ge&   \dist(\phi(v_0),y) - \dist(\phi(v_0),x)   \\
    &\ge&   (r + \tfrac{3k+3}{2}) - (r + 1 + k)     \\
    & = &   \tfrac{k+1}{2}
\end{eqnarray*}
Thus we have that $\dist(\partial A, \partial \phi(A)) \ge \frac{k+1}{2}$ and
by (2) we find that $A$ and $\phi(A)$ merge.

\bigskip

\noindent{\sl Case 1:} $G$ is infinite.

\smallskip

We shall construct a sequence $(\phi_i,S_i)$ where $\phi_i \in Aut(G)$ and
$S_i \in \{\phi_i(L),\phi_i(R)\}$
recursively by the following procedure.  Set $(\phi_{-1},S_{-1}) = (1,L)$ and
$(\phi_0,S_0) = (1,R)$.  For $i \ge 1$ choose $\phi_i \in Aut(G)$ and
$S_i \in \{\phi_i(L),\phi_i(R)\}$ so that the following properties are satisfied
(it follows from claim (3) that such a choice is possible):
\par
\begin{tabular}{lp{4in}}
(i) &   $\phi_i(A)$ merges with $\phi_{i-2}(A)$. \\
(ii) & $S_{i-2} \subseteq \phi_i(A)$. \\
(iii) & $S_i \cap \phi_{i-2}(A) = \emptyset$. \\
(iv) & $\dist(\partial \phi_i(A), \partial \phi_{i-2}(A)) \ge \frac{k+1}{2}$.
\end{tabular}

\noindent{}For every $i \ge 0$, define $X_i = \cup_{j=0}^i \phi_j(A)$.
We now prove by induction that
$X_i$ is a $(k,k+2)$-tube with boundary partition $\{S_{i-1},S_i\}$ for every $i \ge 0$.  It follows immediately from our definitions that $X_0=A$ and
$\{S_{-1},S_0\} = \{L,R\}$, so this is true for $i=0$.  For the inductive step, suppose that this holds for all values less than $i$.
%If $\dist(S_{i-1},S_{i-2}) \le k+1$ or
If $\dist(S_{i},S_{i-1}) \le k+1$, then there is a vertex at distance $\le \lceil \frac{3k+1}{2} \rceil$ from every point in
%$\partial X_{i-1}$ or
$\partial X_i$ and since
$\depth(X_{i}) \ge \depth(X_{i-1})
\ge \depth(X_0) \ge \frac{3k+6}{2} \ge \lceil \frac{3k+5}{2} \rceil$ we have
a contradiction to Lemma \ref{balloon_lem} (with $\lceil \frac{3k+1}{2} \rceil$
playing the role of $k$ in the lemma).
But since $X_{i-1}$ is a $(k,k+2)$-tube and $\dist( \partial \phi_i(A), \partial X_{i-1}) \ge \frac{k+1}{2}$, so by applying Lemma \ref{merge_lem} to the tubes $X_{i-1}$ and $\phi_i(A)$ we conclude that they merge. It follows that $X_i$ is a $(k,k+2)$-tube with boundary partition $\{S_{i-1},S_i\}$.

A straightforward inductive argument now shows that the graph $G$ has two ends.  Furthermore,
$L$ and $R$ are vertex cuts that separate the vertex-set into two sets of infinite size.  Thus, by
Theorem \ref{two_ends_thm} we find that $G$ is $(s,t)$-ring like and
$2st$-cohesive with respect to some cyclic system $\vec{\sigma}$ where $st \le \min \{|L|,|R|\}$, as desired.

\bigskip

\noindent{\sl Case 2:} $G$ is finite.

\smallskip

Let us first define necessary notation for applying voltage graph construction.
For every graph $G$, we define
$$
A(G) = \{(u,v)\in V(G)\times V(G) : \mbox{$u$ and $v$ are adjacent in $G$}\}
$$
and we call the members of $A(G)$ \DEF{arcs}.  We call a map
$\mu : A(G) \rightarrow \ZZ$ a \DEF{voltage} map if $\mu(u,v) = -\mu(v,u)$ for every $(u,v) \in A(G)$.
(This notion extends naturally to general groups, but we have restricted our attention to $\ZZ$ for simplicity.)
For every graph $G$ and voltage map $\mu$, we define a graph ${\mathcal C}(G,\mu)$ as follows:  the vertex-set is $V(G) \times \ZZ$, and vertices $(u,i)$, $(v,j)$ are adjacent if $uv \in E(G)$ and $j-i = \mu(u,v)$.  The map $\pi : V(G) \times \ZZ \rightarrow V(G)$ given by
$\pi(v,g) = v$ is then a covering map, so ${\mathcal C}(G,\mu)$ is a cover of $G$.

If $\mu,\mu'$ are voltage maps on $G$, we say that a mapping
$\Psi : V( {\mathcal C}(G,\mu) ) \rightarrow V({\mathcal C}(G,\mu'))$
\DEF{preserves} $\pi$ if $\pi \circ \Psi = \pi$.  Note that for every integer $j$ the map
$\Psi^j : V(G) \times \ZZ \rightarrow V(G) \times  \ZZ$ given by $\Psi^j(u,i) = (u,i+j)$ is an automorphism
of ${\mathcal C}(G,\mu)$ which preserves $\pi$.  For every $S \subseteq V(G)$ and $m \in \ZZ$ let
$\delta_S^m : A(G) \rightarrow \ZZ$ be the map given by the rule
$$\delta_S^m(u,v) = \left\{ \begin{array}{cl}
    m   &   \mbox{if $u \in S$ and $v \not\in S$}   \\
    -m  &   \mbox{if $u \not\in S$ and $v \in S$}   \\
    0   &   \mbox{otherwise}
    \end{array} \right.$$
We say that two voltage maps $\mu,\mu' : A(G) \rightarrow \ZZ$ are \DEF{elementary equivalent}
if either $\mu' = -\mu$ or $\mu' = \mu + \delta_S^m$ for some $S \subseteq V(G)$ and $m \in \ZZ$.
We say that $\mu$ and $\mu'$ are \DEF{equivalent} and write $\mu \cong \mu'$ if there is a sequence
$\mu = \mu_0, \mu_1, \ldots, \mu_n = \mu'$ of voltage maps on $G$ with $\mu_i$ elementary
equivalent to $\mu_{i+1}$ for
every $0 \le i \le n-1$.  It is straightforward to verify that whenever $\mu \cong \mu'$, there
exists a bijection from ${\mathcal C}(G,\mu)$ to ${\mathcal C}(G,\mu')$ which preserves $\pi$.

By possibly switching $L$ and $R$, we may assume that $|L| \le |R|$.
For every $\phi \in Aut(G)$, define the voltage map
$\mu_{\phi} : A(G) \rightarrow \ZZ$ by the following rule:
\begin{eqnarray*}
\mu_{\phi}(u,v) &=& \left\{ \begin{array}{cl}
    1   &   \mbox{if $u \in \phi(L)$ and $v \in \phi(A)$}   \\
    -1  &   \mbox{if $v \in \phi(L)$ and $u \in \phi(A)$}   \\
    0   &   \mbox{otherwise}
    \end{array} \right. \\
\end{eqnarray*}

We shall use the following properties of voltage maps assigned to tubes.
First of all, if we interchange the roles of $L$ and $R$ in the definition
of $\mu_{\phi}$, we get another voltage map $\mu'_{\phi}$ that is equivalent
to $\mu_{\phi}$ since $\mu'_{\phi} = \mu_{\phi} + \delta^{-1}_{\phi(A)}$.
Second, if tubes $\phi_1(A)$ and $\phi_2(A)$ merge, then
$\mu_{\phi_1}\cong \mu_{\phi_2}$.

Let $\phi_1,\phi_2 \in Aut(G)$ satisfy $\dist(\phi_1(v_0),\phi_2(v_0)) \le 2$, where $v_0$ is the vertex
selected before claim (3).  Then choose
$u \in V(G)$ so that $\dist(u,\phi_i(v_0)) \le 1$ for $i=1,2$ and choose $\psi \in Aut(G)$ so that $\psi(v_{-1}) = u$ (where $v_{-1}$ is as defined before claim (3)).  It follows from claim (3) that $\psi(A)$ and $\phi_i(A)$ merge for $i=1,2$.  It follows from this that $\mu_{\phi_1} \cong \mu_{\psi} \cong \mu_{\phi_2}$.  We conclude that $\mu_{\phi} \cong \mu_1$ for every $\phi \in Aut(G)$.

Let $\tilde{G} = {\mathcal C}(G,\mu_1)$ and for every $i \in {\mathbb Z}$ let
$A_i =  \{ (v,i) : v \in V(G) \}$.  By claim (1) we have that
$G \setminus E[A,L]$ is connected, and it follows that $\{ G[A_i] : i \in {\mathbb Z} \}$
is the set of components of
$\tilde{G} \setminus \{ uv \in E(\tilde{G}) : \mbox{$\pi(u) \in L$ and $\pi(v) \in A$} \}$.
It follows that $\tilde{G}$ has two ends, and that $\kappa_{\infty}(\tilde{G}) \le |L|$.

Let ${\mathcal G} = Aut(G)$ and let
$\tilde{\mathcal G} = \{\phi\in Aut(\tilde{G}) : \mbox{$\phi$ preserves $\pi$} \}$.
Then define the map $\nu : \Tilde{\mathcal G} \rightarrow {\mathcal G}$ by the rule
$\nu(\psi)v = \pi (\psi(v,0))$ ($\nu(\psi)$ is simply the natural projection of $\psi$).
The following diagram now shows the actions of the group ${\mathcal G}$ on
the graph $G$ and of
$\tilde{\mathcal G}$ on $\tilde{G} = {\mathcal C}(G,\mu_1)$.

\begin{center}{$
\begin{diagram}
\node{\tilde{\mathcal G}}   \arrow{e,t}{} \arrow{s,l}{\nu}
\node{\tilde{G}}        \arrow{s,l}{\pi} \\
\node{\mathcal G}   \arrow{e,t}{}
\node{G}
\end{diagram}
$}\end{center}

Next we shall prove that the map $\nu$ is onto.  Let $\phi \in {\mathcal G}$ and define
$\tilde{\phi} : V(G) \times \ZZ \rightarrow V(G) \times \ZZ$ by the rule
$\tilde{\phi}(v,i) = (\phi(v),i)$.  Then $\tilde{\phi}$ is an isomorphism from
${\mathcal C}(G,\mu_1)$ to ${\mathcal C}(G,\mu_{\phi^{-1}})$.  Since $\mu_1 \cong \mu_{\phi^{-1}}$
we may choose an isomorphism
$\psi : {\mathcal C}(G,\mu_{\phi^{-1}}) \rightarrow {\mathcal C}(G,\mu_1)$ which preserves $\pi$.
With these definitions in place, we now have the following commuting diagram.

\begin{center}{$
\begin{diagram}
\node{{\mathcal C}(G,\mu_1)}    \arrow{e,t}{ \hat{\phi} } \arrow{s,l}{\pi}
\node{{\mathcal C}(G,\mu_{\phi^{-1}})}  \arrow{e,t}{\psi} \arrow{s,l}{\pi}
\node{{\mathcal C}(G,\mu_1)}    \arrow{s,l}{\pi}    \\
\node{G}    \arrow{e,t}{\phi}
\node{G}    \arrow{e,t}{1}
\node{G}
\end{diagram}
$}\end{center}

Thus, we have that $\tilde{\phi} \circ \psi \in \tilde{\mathcal G}$ so $\nu$ is onto.
Thus $\tilde{G}$ is vertex transitive, and by Theorem \ref{two_ends_thm} we have that
$\tilde{G}$ is $(s,t)$-ring-like and $2st$-cohesive with respect to some cyclic system
$\vec{\sigma}$ where $st = \kappa_{\infty}( \tilde{G} ) \le |L|$.  Since $\tilde{G}$ is a regular
cover, $\tau = \{ \pi(X) : X \in \sigma \}$ is a partition of $G$.
Since $\nu$ is onto, we conclude that $\tau$ is a system of imprimitivity on $G$.  Now, $\tau$
inherits a cyclic ordering $\vec{\tau}$ from $\vec{\sigma}$, and it follows
that $G$ is $(s,t)$-ring-like and $2st$-cohesive with respect to $\vec{\tau}$.  Since
$st \le |L|$, this completes the proof.
\end{proof}

\section{Main Theorem}

The purpose of this section is to prove our main result, Theorem \ref{main_thm}.
We begin by establishing a lemma on the structure of separations in ring-like graphs.

\begin{lemma}
\label{ring-sep_lem}
Let\/ $G$ be a vertex transitive graph which is $(s,t)$-ring-like and\/
$2st$-cohesive with respect to $\vec{\sigma}$.
Let\/ $A \subseteq V(G)$ and assume that\/ $G[A\cup \partial A]$ is connected,\/
%$G[A]$ is connected,\/
$|A| \le \tfrac{1}{2} |V(G)|$ and\/
$|\partial A| = k$.  Then there exists an interval $J$ of $\vec{\sigma}$
so that the set\/ $Q = \cup_{B \in J} B$ satisfies $A \subseteq Q$
and $|Q \setminus A| \le 2s^2t^2k + 2stk$.
\end{lemma}

\begin{proof}
Let $J_0$ be the set of all $B \in \vec{\sigma}$ with the property that
%$\emptyset \neq A \cap (B \cup B') \neq B \cup B'$ for some $B' \in \vec{\sigma}$ with $B \sim B'$.
$A \cap B \neq \emptyset$ and $A \cap (B \cup B') \neq B \cup B'$ for some $B' \in \vec{\sigma}$ with $B'$ at $\vec{\sigma}$-distance at most 1 from $B$.
Consider a block $B\in J_0$. Let $a\in A\cap B$ and $b\in (B\cup B')\setminus A$.
Since $G$ is $2st$-cohesive with respect to $\vec{\sigma}$, there is a path
in $G$ from $a$ to $b$ of length $\le 2st$. This path contains at least one
vertex in $\partial A$. Since this path starts in $B$ and ends in $B\cup B'$,
its maximum $\vec{\sigma}$-distance from $B\cup B'$ is at most $st^2$.
Therefore, for every $B \in J_0$, there is a vertex $u_B\in \partial A$
contained in a block at $\vec{\sigma}$-distance $\le st^2$ from $B$.

Let us consider all pairs $(B,u)$, where $B\in J_0$ and $u\in\partial A$ is a vertex in a block at $\vec{\sigma}$-distance $\le st^2$ from $B$ and such that there exists a path in $A$ from a vertex $a\in B$ to a neighbor of $u$ of length at most $2st-1$. By the above, each $B\in J_0$ participates in at least one such pair $(B,u_B)$. A vertex $u$ can be the second
coordinate in at most $2st^2+1$ such pairs, since the $\vec{\sigma}$-distance
of $u_B$ from $B$ is at most $st^2$. Let $T$ be a maximal subset of $J_0$
such that the corresponding vertices $u_B$ for $B\in T$ are pairwise different.
Then the set $S = \{u_B\mid B\in T\} \subseteq \partial A$ satisfies
$|S|=|T|\le k$. Further, let
$J_1 = \{ B \in \vec{\sigma} :
\mbox{$B$ is at $\vec{\sigma}$-distance $\le st^2$ from some block $B'$ with $B'\cap S\ne \emptyset$} \}$.
Then $J_0 \subseteq J_1$ and $|J_1| \le |S|(2st^2+1) \le k(2st^2+1)$.
Note further, that any interval of blocks
disjoint from $J_1$ must either all be contained in $A$ or all be disjoint from $A$.  Next, modify the
set $J_1$ to form $J_2$ by adding every block $B$ with the property that the maximal
interval of $\vec{\sigma} \setminus J_1$ containing $B$ has length $\le t$.
Since there are at most $|T|\le k$ such maximal intervals, we have
$|J_2| \le (2st^2+1)k + tk$ and every maximal interval disjoint from $J_2$ has length $\ge t+1$.  Finally,
modify the set $J_2$ to form $J_3$ by adding every block $B$ with $B \subseteq A$.  It follows from the assumption that $G[A\cup\partial A]$ is connected that $J_3$ is an interval of $\vec{\sigma}$.  Furthermore,
setting $Q = \cup_{B \in J_3} B$ we have $A \subseteq Q$ and
$|Q \setminus A| \le s|J_2| \le (2s^2t^2+s)k + stk \le 2s^2t^2k + 2stk$ as desired.
\end{proof}

\begin{lemma}
\label{diam-depth_obs}
Let $G$ be a connected vertex transitive graph and let $A \subset V(G)$ be finite.
Then we have

\begin{enumerate}[label={\rm (\roman{*}) }, ref={\rm (\arabic{*})}]
\item
If\/ $G[A\cup\partial A]$ is connected,
then $\diam(A) < |\partial A|\cdot(2\,\depth(A)+1).$
%If\/ $G[A]$ is connected, then $\diam(A) < |\partial A|\cdot(2\,\depth(A)+1).$
\smallskip
\item
Let $d\ge0,k\ge1,\ell\ge0,m\ge1$ be integers and assume that $|\partial A| \le k$ and $\diam(G) \ge mk(2d+1) + \ell - d + 1$. If\/ $|A| \ge \frac{|V(G)| - 2\ell}{m}$, then $\depth(A) \allowbreak \ge d+1$.
\end{enumerate}
\end{lemma}

\begin{proof}
%Part (i) follows immediately from the observation that $A$ is contained in the union of
%the $|\partial A|$ balls of radius $\depth(A)$, each of which is
%centered at a vertex in $\partial A$,
%together with the assumption that $G[A\cup\partial A]$ is connected.
Let $x,y\in A$. Let $x_0x_1\dots x_t$ be a shortest path in $G[A\cup\partial A]$ from $x=x_0$ to $y=x_t$. For each $i$, there is a vertex $z\in \partial A$ at distance at most $d={\mathcal depth}(A)$ from $x_i$. By the minimality of $t$, each $z\in \partial A$ appears in this way for at most $2d+1$ indices $i\in\{0,1,\dots,t\}$.
Thus $t < |\partial A|(2d+1)$ and the same bound holds for the diameter of $A$.

For part (ii), let $x_0,x_r$ ($r=mk(2d+1)+\ell-d+1$) be vertices at distance
$r$ in $G$, and let $x_0,x_1,x_2,\dots,x_r$ be a shortest path joining them.
Consider the ball $B_0$ of radius $\ell$ centered at $x_0$ and balls
$B_i$ of radius $d$ centered at $x_{\ell+1-d+i(2d+1)}$, for $1\le i\le mk$.
Then $B_0,B_1,\dots,B_{mk}$ are pairwise disjoint. Since $|B_0|\ge 2\ell+1$,
we conclude that every ball of radius $d$ in $G$ contains less than
$\tfrac{1}{km}(|V(G)|-2\ell)$ vertices.
If $\depth(A) \allowbreak \le d$, then balls of radius $d$ centered
at $\le k$ vertices in $\partial A$ would cover $A$, so
$|A| < \frac{|V(G)| - 2\ell}{m}$, contrary to our assumption.
\end{proof}

\begin{lemma}
\label{get-ring_lem}
Let\/ $G$ be a connected vertex transitive graph, let $A \subseteq V(G)$ be finite and set\/
$k = |\partial A|$.  If\/ $\diam(G) \ge 31(k+1)^2$ and
$\depth(A)\ge k+1$ and $\depth(V(G) \setminus (A \cup \partial A)) \ge k+1$, then there
exist integers $s,t$ with $st \le \frac{k}{2}$ and a cyclic system $\vec{\sigma}$ so that $G$ is
$(s,t)$-ring-like and\/ $2st$-cohesive with respect to $\vec{\sigma}$.
\end{lemma}

\begin{proof}
We may assume that $A$ is a set which satisfies the assumptions of the lemma, and further,
is chosen so that

(i) $\,|\partial A|$ is minimum,

(ii) $|A|$ is minimum subject to (i).

Note that (ii) implies that $G[A]$ is connected and that
$|A| \le |V(G) \setminus (A \cup \partial A)|$.  We
proceed with two claims.

\bigskip

\noindent{\bf (1)} $\depth(A) = k+1$.

\smallskip

Suppose (for a contradiction) that $\depth(A) > k + 1$.  Choose a vertex
$x \in A$ at depth $k+2$, its neighbor $y \in A$ at depth $k+1$, and an automorphism
$\phi$ with $\phi(x) = y$.  Set $A_1 = A$ and $A_2 = \phi(A)$ and let the sets
$P,Q,S,T,U,W,X,Y,Z$ be as given in Figure 1.  If
$|S|,|X| \ge \tfrac{1}{4} (|V(G)| - 3k)$ then it follows from part (ii) of
Lemma \ref{diam-depth_obs} (with $m=4$, $d=k$, and $\ell=\lceil\tfrac32 k\rceil$)
that $\depth(S), \depth(X) \ge k+1$ and by uncrossing
that $|\partial S| + |\partial X| \le 2k$.  But then either $S$ or $X$ contradicts our choice of $A$.
Thus, we may assume that either $|S|$ or $|X|$ is less than
$\tfrac{1}{4} (|V(G)| - 3k)$.  If $G$ is finite
and $|S| \le \tfrac{1}{4} (|V(G)| - 3k)$, then
$|S \cup Z| + k \ge |S \cup W \cup Z| \ge \tfrac12 (|V(G)|- k)$ so
$|Z| \ge \tfrac{1}{4} (|V(G)| - 3k)$.  The same conclusion holds under the assumption
that $|X| \le \tfrac{1}{4} (|V(G)| - 3k)$, so by Lemma \ref{diam-depth_obs}(ii) we conclude that
$Z$ has depth $\ge k+1$.  Of course, $Z$ has depth $\ge k+1$ also when
$G$ is infinite.
It follows from our construction that $P$ has depth $\ge k+1$.
Now, by uncrossing, either $|\partial (P \cup Q \cup S \cup T \cup X)| < k$ and this set
contradicts the choice of $A$ for (i), or $|\partial P| \le k$ and this set contradicts the choice
of $A$ for (i) or (ii).

\bigskip

\noindent{\bf (2)} For every $x \in V(G)$ and every $n$,
$k \le n \le \diam(G) - 2k^2 - 2k-1$, there exists a vertex
set $D$ with $|\partial D| = k$ and
$B(x,n) \subseteq D \subseteq B(x,n + 2k^2 + 2k - 1)$.

\smallskip

Note first that this claim holds for $n=k$: by claim (1) we see that
$B(x,k)\subseteq \phi(A)$, where $\phi$ is an automorphism of $G$ mapping a vertex
at depth $k+1$ in $A$ onto $x$. Further, part (i) of Lemma \ref{diam-depth_obs}
shows that $\diam(A) = \diam(\phi(A)) <
k(2\,\depth(A)+1) = 2k^2 + 3k$. Thus,
$\phi(A)\subseteq B(x,3k+2k^2-1)$, so we may take $D=\phi(A)$ when $n=k$.

Suppose (for a contradiction) that such a set does not exist for $n$ and $x$, and choose a
maximal set $C \subseteq V(G)$ with $|\partial C| = k$ and
$B(x,k) \subseteq C \subseteq B(x,n+2k^2 + 2k - 1)$ (such a set exists as shown above).  Choose a shortest path $P$ from $x$ to $\partial C$, let $y$ be the
end of $P$ in $\partial A$, and choose a vertex $z$ on $P$ at distance $k$ from $y$.
Since $B(x,n)\not\subseteq C$, we have that $\dist(x,z) < n-k$.
Using the fact that (2)
holds when $n=k$, choose a set $D \subseteq V(G)$ with $|\partial D| = k$ so that
$B(z,k) \subseteq D \subseteq B(z,2k^2 + 3k-1)$.
Now $V(G) \setminus (C \cup D)$ contains a ball of radius $\ge k+1$
so $\depth(V(G) \setminus (C \cup \partial C \cup D \cup \partial D) \ge k+1$.
If $| \partial (C \cap D)| < k$ then $B(y,k-1) \subseteq C \cap D$ by construction so
$C \cap D$ has depth $\ge k$ and contradicts our choice of $A$ for (i).  Otherwise, it follows from uncrossing (Lemma \ref{uncrossing_lem}) that
$|\partial (C \cup D)| \le k$. Since $\dist(x,z) < n-k$, we have
$C \cup D \subseteq B(x,n+2k^2  + 2k-1)$. Since $y\in D\setminus C$,
we have $|C\cup D| > |C|$, so this contradicts our choice of $C$.
This completes the proof of claim (2).

\bigskip

Set $q = 12k^2 + 15k + 3$ and $h = 6k^2 + 10k$.  Let
$v_{-q}, v_{-q+1}, \ldots, v_{q + k^2 + h + 2}$ be the vertex sequence of a shortest path in $G$.  Next, apply (2) to choose sets $C,D^-,D^+ \subseteq V(G)$
with $|\partial C| = |\partial D^-| = |\partial D^+| = k$ so that the following hold:
\begin{eqnarray*}
B(v_0, q - k^2 - k) \subseteq & C & \subseteq B( v_0, q + k^2 +k-1)         \\
B(v_{-q}, k^2 + 3k) \subseteq & D^- &  \subseteq B(v_{-q}, 3k^2 + 5k-1)   \\
B(v_q, k^2 + 3k) \subseteq &D^+&  \subseteq B(v_q, 3k^2 + 5k-1)
\end{eqnarray*}
Then $B(v_{q - k^2 - 2k}, k) \subseteq C \cap D^+$ and
$B(v_{q + k^2 + 2k},k) \subseteq D^+ \setminus C$ so these sets have depth $\ge k+1$.  Furthermore, $B(v_0,k) \subseteq C \setminus D^+$
and $B(v_{q+ k^2 + h},k) \subseteq V(G) \setminus (C \cup D^+)$ so these sets have depth $\ge k+1$.
It now follows from our choice of $A$ and uncrossing, that
$|\partial (C \cap D^+)| = |\partial (C \setminus D^+)|
= |\partial (D^+ \setminus C)|
= |\partial (C \cup D^+)| = k$.  But then by part (iii) of
Lemma \ref{uncrossing_lem} we find that
$R = (D^+ \cup \partial D^+) \cap \partial C$ satisfies $|R| \ge \frac{k}{2}$.  By a similar argument, we
find that $L = (D^- \cup \partial D^-) \cap \partial C$ satisfies $|L| \ge \frac{k}{2}$.  Thus
$\{L,R\}$ is a partition of $\partial C$.  If $x,y \in R$, then
$x,y \in B(v_q,3k^2 + 5k)$ so $\dist(x,y) \le 6k^2 + 10k = h$.  Similarly if $x,y \in L$, then
$\dist(x,y) \le h$.  If $x \in L$ and $y \in R$, then
$\dist(x,y) \ge \dist(v_{-q},v_{q}) -
\dist(v_{-q},x) - \dist(y,v_q)
\ge 2q - 6k^2 - 10k = 3h+6$.  Thus $C$ is a $(3h+6,h)$-tube with boundary partition
$\{L,R\}$ and $B(v_{q + k^2 + h+2},h+2)$ is disjoint from $C$ so
$\depth(V(G) \setminus (C \cup \partial C)) \ge h+2$.  Applying the tube lemma to this yields the desired conclusion.
\end{proof}

We are prepared to complete the proof of our main theorem.

\begin{proof}[Proof of Theorem \ref{main_thm}]
Let $A \subseteq V(G)$ satisfy the assumptions of the theorem.
Let $A' = V(G) \setminus (A \cup \partial A)$. Observe that $\partial A' = \partial A$.
If $G$ is finite, then by assumption
$|V(G) \setminus (A \cup \partial A)| \ge \frac{1}{2}(|V(G)| - 2k)$.
By using statement (ii) of Lemma \ref{diam-depth_obs} applied to the set $A'$
with $m=2$ and $l=d=k$, we conclude that $\depth(A') \ge k+1$.
The same conclusion holds trivially if $G$ is infinite.
If $\depth(A) \le k$ then by
(i) of Lemma \ref{diam-depth_obs} and Theorem \ref{diam_bound_thm} we find that
$|A| \le 2k^3+k^2$ so case (i) holds. Note that the parenthetical comment in (i)
is a direct application of Theorem \ref{vertex-con_thm}. Otherwise it follows from Lemma \ref{ring-sep_lem} and Lemma \ref{get-ring_lem} that (ii) holds.
\end{proof}

%%%%%%%%%%%%%%%%%%%%%%%%%%%%%%%%%%%%%%%%%%%%%%%%%%%%%%%%%%%%%%%%%%
%
%%%%%%%%%%%%%%%%%%%%%%%%%%%%%%%%%%%%%%%%%%%%%%%%%%%%%%%%%%%%%%%%%%
\section{Proofs of main corollaries}
\label{sect:proofs}

It remains to prove the main corollaries of Theorem \ref{main_thm}.

\begin{proof}[Proof of Theorem \ref{thm:3}]
Let $C=B\setminus\{e\}$, where $e$ is the identity element of the group.
Let $G$ be the Cayley graph of $\mathcal G$ with respect to the symmetric generating set $C$. Observe that $\partial A = BA \setminus A$ and that $A\subseteq BA$. We are going to apply Theorem \ref{main_thm} to the set $A\subset V(G)$. Let us first assume that $G[A\cup \partial A]$ is connected. Since $|BA| < |A| + \tfrac{1}{2}|A|^{1/3}$, and since $A\subseteq BA$, we see that $k=|\partial A| < \frac{1}{2}|A|^{1/3}$. Since $\mathcal G$ is infinite, the assumptions of Theorem \ref{main_thm} are satisfied
and we have one of the outcomes (i) or (ii) of Theorem \ref{main_thm}. If (ii) holds, then let $B_n$ ($n\in\ZZ$) be the blocks of imprimitivity corresponding to the $(s,t)$-ring-like structure of $G$, where $B_0$ contains the identity element of the group. Let
$$
   N = \{ g\in {\mathcal G} \mid B_n g = B_n \text{ for every } n\in \ZZ\}.
$$
Since $\mathcal G$ acts regularly on the vertex set of its Cayley graph $G$ by right multiplication, we conclude that $|N|\le |B_0|=s\le \tfrac{1}{2} k$. It is easy to see that $N$ is a normal subgroup of $\mathcal G$ and that ${\mathcal G}/N$, which is acting transitively on the two-way-infinite quotient, is either cyclic or dihedral. This gives one of the outcomes of the theorem.

Suppose now that we have outcome (i) of Theorem \ref{main_thm}. In that case we conclude, in particular, that $|A|\le 2k^3 + k^2$. If we use the fact that $k<\frac{1}{2}|A|^{1/3}$, we conclude that
$$
   |A| < \tfrac{1}{4}|A| + \tfrac{1}{4}|A|^{2/3} \le \tfrac{1}{2}|A|.
$$
This contradiction completes the proof when $G[A\cup \partial A]$ is connected.

If $G[A\cup \partial A]$ is not connected, then $A$ can be partitioned into sets $A_1,\dots,A_l$ such that the sets $A_i\cup \partial A_i$ ($1\le i\le l$) induce connected subgraphs of $G$ and partition $A\cup\partial A$. Let us define $k_i=|\partial A_i|$ ($1\le i\le l$). We may assume that we have the first outcome of Theorem \ref{main_thm} for each $A_i$, which, as shown above, implies that $k_i \ge \frac{1}{2}|A_i|^{1/3}$.
It follows that
$$
   k = \sum_{i=1}^l k_i \ge \frac{1}{2} \sum_{i=1}^l |A_i|^{1/3} \ge
   \frac{1}{2} \biggl(\sum_{i=1}^l |A_i|\biggr)^{1/3} = \frac{1}{2} |A|^{1/3}.
$$
\end{proof}

%\begin{proof}[Proof of Corollary \ref{cor:1.10extended}]
%\bojan{Missing?}
%\end{proof}

Next we turn towards the proof of Corollary \ref{cor:19}. In this corollary we use the notion of the tree-width. This parameter, which has been introduced in the graph minors theory, formalizes the notion of a graph being ``tree-like''. If $G$ is a graph and $T$ is a tree, then a \DEF{tree-decomposition} of $G$ in $T$ is family of subtrees $T_v\subseteq T$ ($v\in V(G)$), such that whenever $uv\in E(G)$, the corresponding subtrees intersect, $T_v\cap T_u\ne \emptyset$. The \DEF{order} of the tree-decomposition is defined as the maximum cardinality of the sets $Y_t = \{v\in V(G)\mid t\in V(T_v)\}$ taken over all vertices $t\in V(T)$. Finally, the \DEF{tree-width} of $G$ is the minimum order of a tree-decomposition of $G$ minus 1. One can only consider tree-decompositions without ``redundancies'' in which case the set $Y_s\cap Y_t$ is a separator of $G$ for every edge $st\in E(T)$. This explains why the notion of the tree-width is related to the subject of this paper.

We will make use of the following result; see \cite{Re} or \cite[Section 11.2]{FG}. We include a short proof for completeness.

\begin{lemma}[Balanced Separator Lemma]
\label{lem:balanced sep}
If a graph\/ $G$ has tree-width less than $k$, then for every vertex set $W\subseteq V(G)$ there exists a set
$S\subseteq V(G)$ with $|S|\le k$ such that every connected component of\/ $G-S$ contains at most $\tfrac{1}{2}|W|$ vertices from $W$.
\end{lemma}

\begin{proof}
Let us consider a tree-decomposition of $G$ in a tree $T$ of order at most $k$, and assume that subject to these conditions, $T$ is minimum possible. Let $st\in E(T)$ be an edge of $T$. If $Y_t\subseteq Y_s$, then we could replace the tree $T$ by the smaller tree $T'=T/st$ obtained by contracting the edge $st$, and replace each $T_v$ containing the edge $st$ by the subtree $T_v'=T_v/st\subseteq T'$ to obtain a tree-decomposition of the same order and having smaller tree. It follows that the set $Y_t\cap Y_s$ separates the graph $G$, i.e. $G-(Y_t\cap Y_s)$ is disconnected. Let us observe that this also implies that for each leaf $t$ in $T$, there is a vertex $v\in V(G)$ such that $T_v = \{t\}$.

Suppose that $st\in E(T)$ and that there is a component $C$ of $G-(Y_t\cap Y_s)$ containing more than $\tfrac{1}{2}|W|$ vertices from $W$. Let $T_1$ and $T_2$ be the subtrees of $T-st$ containing the vertex $s$ and $t$, respectively. If $V(C)\subseteq \cup_{p\in V(T_1)} Y_p$, then we orient the edge $st$ from $t$ to $s$, otherwise from $s$ to $t$. By repeating this for all edges of $T$, we get a (partial) orientation of the edges of $T$ with the property that the out-degree of each vertex is at most 1. Therefore, there is a vertex $t\in V(T)$ whose out-degree is 0. Clearly, the set $S=Y_t$ satisfies the conclusion of the lemma.
\end{proof}

\begin{proof}[Proof of Corollary \ref{cor:19}]
Let $G$ be a connected finite vertex transitive graph and suppose that the tree-width of $G$ is less than $k$.
Let us consider a tree-decomposition of $G$ in a tree $T$ of order at most $k$, and assume that subject to these conditions, $T$ is minimum possible. Let $t\in V(T)$ be a leaf of $T$. As shown in the proof of the Balanced Separator Lemma, there is a vertex $v\in V(G)$ such that $T_v = \{t\}$. All neighbors of this vertex are in $Y_t$ and $|Y_t|\le k$, hence the degree of $v$ in $G$ is at most $k-1$. If the diameter of $G$ is less than $31(k+1)^2$, then we have the last outcome of the corollary, so we may assume henceforth that the diameter is at least $31(k+1)^2$.

Let $W$ be a path in $G$ joining two vertices that are at distance $31(k+1)^2$. By the Balanced Separator Lemma \ref{lem:balanced sep}, there exists a set $S\subset V(G)$ with $|S|\le k$ such that no component of $G-S$ contains more than $\tfrac{1}{2}|W|$ vertices of $W$. For each vertex $u\in S$, let $A_u$ be the vertex set consisting of all vertices in $W$ whose distance from $u$ in $G$ is at most $k$. Let $W_u$ be the smallest segment on $W$ that contains $A_u$. As any two vertices in $A_u$ are at distance at most $2k$ from each other, we have $|W_u|\le 2k+1$. Since
$$
   |W\setminus (\cup_{u\in S} W_u)| \ge |W| - k(2k+1) > \tfrac{1}{2}|W|
$$
there are at least two components of $G-S$ that contain a vertex in $W\setminus (\cup_{u\in S} W_u)$. Let $A$ be the vertex set of the smaller one of these two components. Then $|A|<\tfrac{1}{2}|V(G)|$ and $\partial A\subseteq S$, thus $|\partial A|\le k$. Also, $G[A]$ is connected and since $A$ contains a vertex in $W\setminus (\cup_{u\in S} W_u)$, we have that $\depth(A) > k$. By Theorem \ref{main_thm} we conclude that $G$ is $(s,t)$-ring-like with $2st \le k$, i.e., we have the outcome (i) of Corollary \ref{cor:19}. This completes the proof.
\end{proof}

\end{document}